\RequirePackage[2020-10-01]{latexrelease}
\documentclass{article}
\usepackage{hyperref}
\hypersetup{	
    colorlinks,
    citecolor=magenta,
    filecolor=black,
    linkcolor=blue,
    urlcolor=black
}
\usepackage[UKenglish]{babel}
\usepackage{amsmath,amssymb,enumerate,latexsym,makeidx,xcolor,amsthm,cleveref,graphicx,dsfont}

\usepackage{verbatim}

\usepackage{old-arrows}
\usepackage{bbm}
\makeindex

\catcode`\>=\active \def>{
\fontencoding{T1}\selectfont\symbol{62}\fontencoding{\encodingdefault}}
\newcommand{\assign}{:=}

\newcommand{\nin}{\not\in}
\newcommand{\nobracket}{}

\newcommand{\mathd}{\mathrm{d}}
\newcommand{\mathe}{\mathrm{e}}
\newcommand{\suchthat}{:}

\newcommand{\tmaffiliation}[1]{\\ #1}
\newcommand{\tmem}[1]{{\em #1\/}}
\newcommand{\tmop}
[1]{\ensuremath{\operatorname{#1}}}

\newcommand{\tmstrong}[1]{\textbf{#1}}

\newenvironment{descriptioncompact}{\begin{description} }{\end{description}}
\newenvironment{enumeratealpha}{\begin{enumerate}[a{\textup{)}}] }{\end{enumerate}}
\newenvironment{itemizedot}{\begin{itemize} }{\end{itemize}}
\renewenvironment{proof}{\noindent\textbf{Proof.\ }}{\hspace*{\fill}$\Box$\medskip}

\newcommand{\orange}[1]{{\color{orange}#1}}

\renewcommand{\phi}[0]{\varphi}
\numberwithin{equation}{section}

\renewcommand{\phi}[0]{\varphi}

\newcommand{\N}[0]{\mathbb{N}}

\DeclareMathOperator{\dom}{dom}

\theoremstyle{definition}
\newtheorem{definition}{Definition}[section]
\newtheorem{notation}[definition]{Notation}
\newtheorem{example2}[definition]{Example}
\newenvironment{example}{\begin{example2}}{\qed\end{example2}}
\newtheorem{remark2}[definition]{Remark}
\newenvironment{remark}{\begin{remark2}}{\qed\end{remark2}}
\theoremstyle{plain}
\newtheorem{theorem}[definition]{Theorem}
\newtheorem{proposition}[definition]{Proposition}
\newtheorem{lemma}[definition]{Lemma}
\newtheorem{corollary}[definition]{Corollary}

\newcommand{\Addresses}{{
		\bigskip
		\footnotesize
		
		V.~Bagayoko, \textsc{Institut de Mathématiques de Jussieu -- Paris Rive Gauche, 75013 Paris, France}\par\nopagebreak
		\textit{E-mail address}: \texttt{bagayoko@imj-prg.fr}
		
		\medskip
		
		L.~S.~Krapp, \textsc{Institut für Interdisziplinäre Sprachevolutionswissenschaft, 8050 Zürich, Switzerland \& Fachbereich Mathematik und Statistik, Universität Konstanz, 78457 Konstanz, Germany}\par\nopagebreak
		\textit{E-mail address}: \texttt{sebastian.krapp2@uzh.ch}
		
		\medskip
		
		S.~Kuhlmann, \textsc{Fachbereich Mathematik und Statistik, Universität Konstanz, 78457 Konstanz, Germany}\par\nopagebreak
		\textit{E-mail address}: \texttt{salma.kuhlmann@uni-konstanz.de}
		
		\medskip
		
		D.~Panazzolo, \textsc{Université de Haute-Alsace, Laboratoire de Mathématiques, Informatique et Applications, 68093 Mulhouse, France}\par\nopagebreak
		\textit{E-mail address}: \texttt{daniel.panazzolo@uha.fr}
		
		\medskip
		
		M.~Serra, \textsc{Fakultät für Mathematik, Technische Universität Dortmund, 44227 Dortmund, Germany}\par\nopagebreak
		\textit{E-mail address}: \texttt{michele.serra@tu-dortmund.de}
		
}}

\begin{document}

\title{Automorphisms and derivations on algebras endowed with formal infinite sums}

\author{
  Vincent Bagayoko
  \tmaffiliation{IMJ-PRG}
  \and
  Lothar Sebastian Krapp
  \tmaffiliation{Universität Zürich}
  \and
  Salma Kuhlmann
  \tmaffiliation{Universität Konstanz}
  \and
  Daniel Panazzolo
  \tmaffiliation{Université de Haute Alsace}
  \and
  Michele Serra
  \tmaffiliation{TU Dortmund}
}

\maketitle

\begin{abstract} 
    We establish a correspondence between automorphisms and derivations on certain algebras of generalised power series. In particular, we describe a Lie algebra of derivations on a field $k(\!(G)\!)$ of generalised power series (where $k$ is a field of characteristic $0$ and $G$ is a linearly ordered abelian group), 
    exploiting our knowledge of its group of valuation preserving automorphisms. The correspondence is given by the formal Taylor expansion of the exponential. In order to define the exponential map, we develop an appropriate notion of summability of infinite families in algebras. We show that there is a large class of algebras in which the exponential induces the above correspondence. 
\end{abstract}


\section*{Introduction}

\addcontentsline{toc}{section}{Introduction}

    Let $k$ be a field of characteristic $0$. 
    The automorphism group of the valued field $k(\!(t)\!)$ of Laurent series was studied in Schilling's classical paper \cite{schilling}. A derivation on $k(\!(t)\!)$ is a $k$-linear map $\partial\colon k(\!(t)\!) \longrightarrow k(\!(t)\!)$ satisfying the Leibniz product rule.
    The $k$-vector space of derivations on $k(\!(t)\!)$ becomes a Lie algebra once endowed with the Lie bracket $[\partial_1,\partial_2]:=\partial_1\circ\partial_2-\partial_2\circ\partial_1$. Let $\sigma$ and $\partial$ denote respectively  an automorphism and a derivation. Via the Taylor series of the logarithm and of the exponential,
    \begin{equation}
	\log (\sigma) = \sum_{n \geqslant 1} \frac{(-1)^{n+1}}{n} (\tmop{Id}-\sigma)^{[n]} \text{  and  } \exp (\partial) = \sum_{n \geqslant 0} \frac{1}{n!} \partial^{[n]} \label{eq-exp-log-intro}
   \end{equation}
    (where $\varphi^{[n]}$ denotes the $n^{\text{th}}$ iterate of a map $\varphi$), one obtains the fundamental relation between the group of automorphisms and the Lie algebra of derivations of $k(\!(t)\!)$ (cf.\ \cite[Chapter~3]{reutenauer}).
    
    Fields of generalised power series $k(\!(G)\!)$ with coefficients in $k$ and exponents in an arbitrary linearly ordered abelian group $G$ are instrumental to the valuation and model theory of fields (for example, when $G=\mathbb{Z}$, then $k(\!(G)\!)$ is just the Laurent series field).
    In \cite{KuhlSer:vaut}, the authors described the group of valuation preserving automorphisms of $k(\!(G)\!)$, as a semidirect product of four distinct factors.

    Our initial motivation for the present work was to describe the Lie algebra of derivations on $k(\!(G)\!)$, exploiting our knowledge of its group of valuation preserving automorphisms.
    To this end, we cast the problem into the following more general setting. Given a $k$-algebra $A$ (not necessarily commutative), we study the interplay between the Lie algebra of its derivations and the group of its automorphisms. Here, the essential issue is to give appropriate conditions under which the infinite sums in \eqref{eq-exp-log-intro} are indeed well-defined. To deal with this issue, we explore a general notion of summability in algebras, and of strongly linear maps (that is, linear maps that commute with infinite sums). We show that there is a large class of algebras in which the exponential and logarithm induce the desired correspondence for strongly linear $\partial$ and $\sigma$. 

    In \Cref{section-strong-vector-spaces} we introduce the axiomatic notion of vector spaces and algebras with a structure of formal summability, called summability spaces and summability algebras (see Definitions~\ref{def:summabilityspace} and~\ref{def:summabilityalgebra}). A related notion has been presented independently in a categorical framework by Freni \cite{Fre:strong}. We show that the algebra of strongly linear maps on a summability space inherits a natural structure of summability algebra (\Cref{prop:Linalgebra}). Our main collection of examples of summability spaces and algebras is that of spaces and algebras of Noetherian series, introduced in \Cref{subsection:Noetherianseries}. Those include algebras of polynomials, algebras of formal series in commuting or non-commuting variables, and fields of generalised power series $k(\!(G)\!)$. In \Cref{subsection-kJJ}, we recall the construction of the algebra of Noetherian series $k  \langle\! \langle J \rangle\! \rangle$, which is the algebra of formal series with coefficients in $k$ and non-commuting variables $ X_j, j \in J$. In \Cref{subsection-evaluations}, we introduce the notion of summability algebras with evaluations (see \Cref{def-closed}). This allows any power series in $k  \langle\! \langle J \rangle\! \rangle$ to be evaluated in such an algebra $A$ for each set $J$ (see \Cref{rem-evaluations}). More precisely, the algebra $A$ is local (\Cref{prop-closed-maximal-ideal}), and each summable family $\mathbf{a}=(\mathbf{a}_j)_{j \in J}$ ranging in its maximal ideal $\mathfrak{m}$ defines a unique strongly linear algebra morphism $\tmop{ev}_{\mathbf{a}}\colon k  \langle\! \langle J \rangle\! \rangle \longrightarrow A$ that maps each variable $X_j$ to $\mathbf{a}_j$.
 
    In \Cref{section-closed}, we apply evaluation morphisms for the univariate and bivariate cases only. In the univariate case (when $J$ is a singleton), we can identify
    $k  \langle\! \langle J \rangle\! \rangle$ with $k[\![X]\!]$ (the algebra of formal power series in the variable $X$). In particular, the formal power series
	\[ \log (1 + X) \assign \sum_{n
		\geqslant 1} \frac{(- 1)^{n + 1}}{n} X^n
	\text{ and }\exp (X) \assign \sum_{n \geqslant 0} \frac{1}{n!} X^n \]
    can be evaluated at each $\varepsilon \in \mathfrak{m}$.   Furthermore, the relations
	\[ \exp (\log (1+\varepsilon)) = 1+\varepsilon \ \text{ \  and   }\ \log
	(\exp (\varepsilon)) = \varepsilon \]
    follow (see \Cref{cor-exp-log-bij}) from evaluating the corresponding identities {\cite[Part I, Section~IV.7, Theorem~7.2]{Serre:Lie} in $k[\![X]\!]$ at $\varepsilon$. In the bivariate case, we likewise obtain (see \Cref{cor-BCH-operation}) the Baker--Campbell--Hausdorff formula (see \cite[Section~5.2]{Hall:Lie}
    or \cite[page~29]{Serre:Lie})
    \[ \log (\exp (\varepsilon) \cdot \exp (\delta)) = \varepsilon + \delta + \frac{1}{2}  [\varepsilon, \delta] + \frac{1}{12}  ([\varepsilon, [\varepsilon, \delta]] - [\delta, [\varepsilon, \delta]]) + \cdots \] for $\varepsilon,\delta \in \mathfrak{m}$. Our first main result is a correspondence between automorphisms and derivations (\Cref{th-der-aut-correspondence}) for all summability subalgebras (of
    strongly linear maps) with evaluations. 

    The purpose of \Cref{section-Noetherian-series} is to apply \Cref{th-der-aut-correspondence} to algebras of Noetherian series.  Given an algebra $\mathbb{A}$ of Noetherian series, we consider the algebra $k\tmop{Id}_{\mathbb{A}} + \mathfrak{m}$ where $\mathfrak{m}$ is the closed ideal of contracting strongly linear maps
    (see \Cref{def-contracting}). We show (\Cref{th-superclosed-End}) that this is an algebra with evaluations. Our second main result \Cref{th-der-aut-correspondence-Hahn-field} is that the exponential is a bijection between the subset of $\mathfrak{m}$ consisting of derivations and the subset of $\tmop{Id}_{\mathbb{A}}+\mathfrak{m}$ consisting of automorphisms of $\mathbb{A}$. In the case when $\mathbb{A}=k(\!(G)\!)$ is a field of Hahn series, this group is one of the three factors in the decomposition \cite{KuhlSer:vaut} of strongly linear valuation preserving automorphisms of $k(\!(G)\!)$ (\Cref{rem:1aut}). It follows in particular (\Cref{cor-divisible}) that this group is divisible and torsion-free. We derive from \Cref{th-der-aut-correspondence-Hahn-field} our third and final main result \Cref{th-Lie}, which moreover takes into account the Lie structure on the corresponding algebras. More precisely, we prove a formal analog of the Lie homomorphism theorem (cf.\ \cite[Theorem 5.6]{Hall:Lie}) for contracting strongly linear derivations and strongly linear automorphisms.

    In the \hyperref[Appendix]{Appendix}, we focus on the case where $k$ is an ordered exponential field (see \cite{Kuhl:Buch}), and we discuss the possibility of extending our correspondence to the other two factors in the semidirect decomposition of the group of valuation preserving automorphisms of $k(\!(G)\!)$.

\subsection*{Conventions and notations}

\addcontentsline{toc}{subsection}{Conventions and notations}

    \textbf{Throughout this work, we fix a field $k$. All vector spaces and algebras considered below are over $k$. All algebras are associative}, but they are not necessarily unital, nor commutative. 
    Whenever $k$ has to be of characteristic $0$, this will be specifically pointed out.
	
    We denote by $\N$ the set of natural numbers \emph{with} $0$ and by $\N^{>0}$ \emph{without} $0$.
    Let $X$ be a set. 
    The power set of $X$ is denoted by $\mathcal{P}(X)$ and its subset consisting of all finite subsets of $X$ is denoted by $\mathcal{P}_{\mathrm{fin}}(X)$. 
    Given some element $p\in X$, we let $\mathds{1}_p^X$ be the indicator map from $X$ to $k$, i.e.\ the map $\mathds{1}_p^X\colon X\longrightarrow k$ mapping $p$ to $1$ and everything else to $0$. If the set $X$ is clear from the context, we simply write $\mathds{1}_p$ instead of $\mathds{1}_p^X$.

    Given sets $A$ and $B$, we identify a map $f\colon A \longrightarrow B$ with its graph $\{(a,f(a)):a\in A\}\subseteq A\times B$ as well as with the family $(f(a))_{a\in A}$.
    The set of maps from $A$ to $B$ is denoted by $B^A$\label{autolab2}. If $B \subseteq C$, then $B^A \subseteq C^A$. Note also that~$A^{\emptyset} = \{\emptyset\}$.

    Recall that a {\textbf{partial ordering}}{\index{ordering}} on a set $\Omega$ is a binary relation $<$ on $\Omega$ such that for any $p, q, r \in \Omega$, we have
	\[ p \nless p \text{ and }[(p < q \wedge q < r)
		   \Longrightarrow p < r].\] 

    For subsets $A,B \subseteq \Omega$, we write $A<B$ if $a<b$ for all $(a,b) \in A \times B$. Likewise, for $a \in A$, we write $a<B$ if $\{a\} < B$ and $B<a$ if $B<\{a\}$.
     A \textbf{linear ordering} on $\Omega$ is a partial ordering $<$ additionally satisfying $$p<q\vee p=q\vee q<p$$ for any $p,q\in \Omega$.
     We say that $(\Omega, <)$ is an \textbf{partially ordered set} (respectively \textbf{linearly ordered set}) if $<$ is a partial ordering (respectively linear ordering) on $\Omega$. We consider subsets of a partially ordered set as partially ordered with the induced ordering. 
     For $p, q \in \Omega$, we write $p \leqslant q$ if $p = q$ or $p < q$. 
     If $u\colon \mathbb{N} \longrightarrow \Omega$ is a sequence, then we say that $u$ is {\bf{increasing}}{\index{increasing, strictly increasing sequence}} (respectively {\bf{strictly increasing}}) if for all $m, n \in \mathbb{N}$ with $m < n$, we have $u (m) \leqslant u (n)$ (respectively $u (m) < u (n)$). A {\bf{subsequence}}{\index{subsequence}} of $u$ is a sequence $v = u \circ \varphi\colon \mathbb{N} \longrightarrow \Omega$ where $\varphi\colon \mathbb{N} \longrightarrow \mathbb{N}$ is strictly increasing.

    If $I$ is a set and $(M, +, 0)$ is a {\bf{monoid}}{\index{monoid}} (i.e.\ an associative, unital magma), then the {\bf{support}}{\index{support of a monoid-valued map}} of $\mathbf{f} \in M^I$ is the subset\label{autolab3} \[ \tmop{supp} \mathbf{f} \assign \{ i \in I \suchthat \mathbf{f} (i) \neq 0 \} \] of $I$.
    
    Given a vector space $V$ and a set $I$, the set $V^I$ of maps $\mathbf{v}\colon I \longrightarrow V$ is equipped with its natural vector space structure (pointwise operations). Note that the subset $V^{(I)}$ of maps with finite support is a subspace of $V^I$. Given two vector spaces $V$ and $W$, we write $\tmop{Lin} (V, W)$ for the vector space of linear maps $V \longrightarrow W$, and we denote the space $\tmop{Lin} (V, V)$ of endomorphisms on $V$ by $\tmop{Lin} (V)$. A {\emph{unital}} algebra $A$ is called {\bf{local}}{\index{local algebra}} if one of the following equivalent assertions is satisfied {\cite[Theorem~19.1]{Lam:non-com}}:
    
    \begin{itemizedot}
        \item $A$ has a unique maximal left ideal;
  
        \item $A$ has a unique maximal right ideal;
  
        \item the set $A\setminus U(A)$ of non-units in $A$ is an ideal of $A$.
    \end{itemizedot}
    
    In this case, the maximal left and right ideals are equal to $A \setminus U(A)$. Given an algebra $(A, +, \cdot)$ and $a, b \in A$, we write\label{autolab4} \[ [a, b] \assign a \cdot b - b \cdot a \in A, \] and we recall that $(A, +, [\cdot, \cdot])$ is a Lie algebra.

\section{Formal summability in algebras}\label{section-strong-vector-spaces}

\subsection{Summability spaces}\label{subsection-strong-vector-spaces}

    Freni {\cite[Chapter 2]{Fre:strong}} introduced a category $\Sigma \tmop{Vect}$ whose objects are vector spaces equipped with a notion of formal sums. We propose an axiomatic description of the same spaces (see \cite[Theorem 2.2.26]{Fre:strong}}), in the vein of {\cite[Section~6.2]{vdH:phd}}, that is more tailored to our purposes.
    
    Let $V$ be a vector space. We generalise the notion of finite summation operators \[ 
				\Sigma_n \colon V^n\longrightarrow V, \  (v_0,\ldots,v_{n-1}) \mapsto v_0+\cdot \cdot \cdot + v_{n-1}
				\] 
    to abstract sums $\Sigma_I \mathbf{v} \in V$ of families $\mathbf{v} \colon I \longrightarrow V$ indexed by possibly infinite sets $I$. Not all families $\mathbf{v}$ in the space $V^I$ can be summed in a consistent way: if $I$ is infinite and $v \in  V \setminus \{0\}$, then the constant family $\mathbf{v}=(v)_{i \in I}$ cannot be summed. We thus have to introduce axioms specifying the intended properties of sets $\dom \Sigma_I \subseteq V^I$ of summable families in conjunction with properties of summation operators $\Sigma_I$. 

    \begin{definition}\label{def:summabilityspace}
        For any set $I$, let $\Sigma_I$ be a $k$-linear map\label{autolab5} \[ \Sigma_I\colon \tmop{dom} \Sigma_I \longrightarrow V \] whose domain $\tmop{dom} \Sigma_I$\label{autolab6} is a subspace of the vector space $V^I$. 
        We say that $\Sigma = (\Sigma_I)_I$ (where $I$ ranges in the class of all sets)\label{autolab8} is a
        {{{\bf{summability structure}}}}{\index{summability structure}} on $V$, or
        that $(V, \Sigma)$ is a {{{\bf{summability
        space}}}}{\index{summability space}}, if $(V, \Sigma)$
        satisfies the axioms \hyperref[ss1]{(\textbf{SS1})} to \hyperref[ss4]{\textbf{(SS4)}} below for any sets $I$ and $J$.

        \begin{itemize}
            \item[($\textbf{SS1}$)] \label{ss1}$V^{(I)} \subseteq \tmop{dom} \Sigma_I$ and $\Sigma_I \mathbf{v} = \sum_{i \in \tmop{supp} \mathbf{v}} \mathbf{v} (i)$ for any $\mathbf{v} \in V^{(I)}$.
  
            \item[($\textbf{SS2}$)] \label{ss2} If $\varphi\colon I \longrightarrow J$ is a bijective map and $\mathbf{v} \in \tmop{dom} \Sigma_J$, then $\mathbf{v} \circ \varphi \in \tmop{dom} \Sigma_I$, and $\Sigma_I (\mathbf{v} \circ \varphi) = \Sigma_J \mathbf{v}$.
  
            \item[($\textbf{SS3}$)] \label{ss3} If $I = \bigsqcup_{j \in J} I_j$\label{autolab7} for some family of pairwise disjoint sets $(I_j)_{j\in J}$ and $\mathbf{v} \in \tmop{dom} \Sigma_I$, then writing $\mathbf{v}_j \assign \mathbf{v}\!\!\upharpoonright_{I_j}$ for each $j \in J$, we have
  
            \begin{descriptioncompact}
                \item[($\textbf{SS3a}$)] \label{ss3a}$\mathbf{v}_j \in \tmop{dom} \Sigma_{I_j}$ for any $j \in J$,
            
                \item[($\textbf{SS3b}$)] \label{ss3b}$( \Sigma_{I_j} \mathbf{v}_j )_{j \in J}  \in \tmop{dom} \Sigma_J$, and
            
                \item[($\textbf{SS3c}$)] \label{ss3c}$\Sigma_I \mathbf{v} = \Sigma_J \left( ( \Sigma_{I_j} \mathbf{v}_j )_{j \in J} \right)$.
            \end{descriptioncompact}
  
            \item[($\textbf{SS4}$)] \label{ss4} If $I = I_1 \sqcup I_2$ for some sets $I_1$, $I_2$ and $(\mathbf{v}, \mathbf{w}) \in \tmop{dom} \Sigma_{I_1} \times \tmop{dom} \Sigma_{I_2}$, then the map $(\mathbf{v} \sqcup \mathbf{w})\colon I \longrightarrow V$ given by $(\mathbf{v} \sqcup \mathbf{w}) (i_1) \assign \mathbf{v} (i_1)$ and $(\mathbf{v}  \sqcup  \mathbf{w}) (i_2) = \mathbf{w} (i_2)$ for any $i_1 \in I_1$ and $i_2 \in I_2$ lies in $\tmop{dom} \Sigma_I$.

        \end{itemize}

        \noindent If, moreover, axiom \hyperref[uss]{(\textbf{UF})} below is satisfied for any $I$, then the summability space $(V,\Sigma)$ is called \textbf{ultrafinite}{\index{ultrafiniteness}}.

        \begin{itemize}
            \item[($\textbf{UF}$)] \label{uss}For any $\mathbf{v} \in \tmop{dom} \Sigma_I$ and     any family $f=(f_i)_{i \in I}$ of $k$-valued maps              $f_i\colon X_i\longrightarrow     k$     with \emph{finite} domains $X_i$,       writing $I f := \{ (i, x) \colon i \in I \wedge     x \in X_i \}$, the family $f\mathbf{v} \assign (f_i (x) \mathbf{v} (i))_{(i,x) \in     If}$ lies in $\tmop{dom} \Sigma_{I f}$.
        \end{itemize}

        \noindent For a summability space $(V,\Sigma)$ and any set $I$, a family $\mathbf{v}\in V^I$ is called \textbf{summable} (in $(V,\Sigma)$) if $\mathbf{v}\in \dom \Sigma_I$.

    \end{definition}

    \begin{remark}
        \begin{enumerate}[(i)]
            \item For our later application, it is convenient to let the index sets $I$ in \Cref{def:summabilityspace} range over all sets. While on a formal set theoretic level, this will entail working with proper classes, in practice it will simply lead to an ease of notation without resulting in any foundational issues.

            \item The reader can see that these axioms generalise various properties of finite sums, including commutativity (or invariance under reindexing) \hyperref[ss2]{\textbf{(SS2)}} and associativity \hyperref[ss3]{\textbf{(SS3)}} of the sum.
    
            \end{enumerate}    
            
    \end{remark}


    \textbf{Throughout the rest of \Cref{section-strong-vector-spaces}, let $(V,\Sigma)$ denote a summability space.}
 
    \begin{notation}
        For a set $I$ and a map $\mathbf{v}\in V^I$ we also  write
        $\sum_{i \in I} \mathbf{v} (i)$ for $\Sigma_I \mathbf{v}$.
    \end{notation}

    Our first lemma establishes further associativity and distributivity results.
    
    \begin{lemma}\label{prop-union-sum}
        Let $I$, $I_1$ and $I_2$ be sets, let $(X_i)_{i\in I}$ be a family of finite sets, and for each $i\in I$, let $f_i\colon X_i\longrightarrow k$.
        Then the following hold:
		\begin{enumerate}[(i)]
			
			\item \label{prop-union-sum:1} If $I = I_1 \sqcup I_2$, then for any $(\mathbf{v}, \mathbf{w}) \in
			\tmop{dom} \Sigma_{I_1} \times \tmop{dom} \Sigma_{I_2}$ we have  
			$\Sigma_I  (\mathbf{v} \sqcup \mathbf{w}) =
			\Sigma_{I_1} \mathbf{v} + \Sigma_{I_2} \mathbf{w}$.
			
			\item  \label{rem-ultrafinite:1} 
			If $(V,\Sigma)$ is ultrafinite, then for any $\mathbf{v}\in \Sigma_I$ also $\big( ( \sum_{x \in
					X_i} f_i(x) ) \mathbf{v} (i) \big)_{i \in I} \in \Sigma_I$ and 
				\[ 
				\Sigma_{I f} (f \mathbf{v})=\sum_{i\in I} \left(\sum_{x \in
					X_i} f_i(x)\right ) \mathbf{v} (i).
				\]

		\end{enumerate}
	\end{lemma}

	\begin{proof}
		\begin{enumerate}[(i)]
				
			\item 	Let $\mathbf{u}=(\mathbf{v}\sqcup \mathbf{w})$. Then $\mathbf{u}_1=\mathbf{v}$ and $\mathbf{u}_2 = \mathbf{w}$. By \hyperref[ss1]{\textbf{SS1}}, \hyperref[ss3]{\textbf{SS3}} and \hyperref[ss4]{\textbf{SS4}},$$
			\Sigma_I  \mathbf{u} = \Sigma_{\{1,2\}} (\Sigma_{I_j}\mathbf{u}_j)_{j\in \{1,2\}} =\Sigma_{I_1}\mathbf{u}_1 + \Sigma_{I_2}\mathbf{u}_2 =   \Sigma_{I_1} \mathbf{v} + \Sigma_{I_2} \mathbf{w}.
			$$
			
			\item 
			For any $i\in I$, let $I_j = \{i\} \times X_i$. Then $I f= \bigsqcup_{i \in I} I_i$. 
			Applying \hyperref[ss3]{\textbf{SS3}}, \hyperref[uss]{\textbf{UF}} and \hyperref[ss1]{\textbf{SS1}}, we obtain 
			\begin{align*}\Sigma_{I f}(f \mathbf{v}) &= \Sigma_I \left( (
			\Sigma_{I_i} (f_i(x)\mathbf{v}(i))_{(i,x)\in \{i\}\times X_i} )_{i \in I} \right)=\sum_{i\in I}  \left(
			\sum_{x\in X_i} f_i(x)\mathbf{v}(i)\right)\\&= \sum_{i\in I} \left(\sum_{x \in
				X_i} f_i(x)\right ) \mathbf{v} (i).\end{align*}
		\end{enumerate}

    \end{proof}

    We now give some examples of summability spaces.

    \begin{example}
    \label{ex-minimal-strong}
        Any vector space $V$ can be endowed with a summability
        structure $\Sigma^{\min}$ called the {\bf{minimal summability structure}}{\index{minimal summability
        structure}}. It is given by 
        \[ \tmop{dom} \Sigma_I^{\min} = V^{(I)}\text{ and }\Sigma_I^{\min} \mathbf{v} =
        \sum_{i \in \tmop{supp} \mathbf{v}} \mathbf{v} (i) \]
        for any set $I$ and any  $\mathbf{v} \in V^{(I)}$.
        The summability space $(V,\Sigma^{\min})$ is ultrafinite. \end{example}

    \begin{example}
    \label{ex-strong-Banach-space}
        Assume that $k =\mathbb{R}$ or $k
        =\mathbb{C}$. Let $\left( V, \left| \hspace{0.8em} \right| \right)$ be a
        Banach space with norm $\left| \hspace{0.8em} \right|\colon V
        \longrightarrow \mathbb{R}^{\geqslant 0}$. For any set $I$, define
        $\tmop{dom} \Sigma_I$ to be the set of families $\mathbf{v} \colon I \longrightarrow V$
        that either have finite support or have countably infinite support such that, given a bijection $i \colon \N
        \longrightarrow \tmop{supp} \mathbf{v}$, the series
        $\sum_{n=0}^\infty | \mathbf{v} (i (n)) |$ converges.
        For $\mathbf{v}\in V^I$ with countably infinite support, we then define $\sum_I \mathbf{v} = \sum_{n=0}^\infty \mathbf{v} (i
        (n))$. Note that this sum does not depend on the choice of bijection $i$.
        
        This summability structure is {\tmem{not}} ultrafinite in
        general. Indeed if $V$ is non-trivial, then for any non-zero $v \in V$, the
        family $\left( \frac{v}{n!} \right)_{n \in \mathbb{N}}$ is summable, whereas
        $(vn!)_{n \in \mathbb{N}}$ is not.
    \end{example}

    We next introduce our main motivating example for summability spaces.

    \begin{example}
    \label{ex-Hahn-space} 
        Let $(\Gamma,<)$ be a non-empty linearly ordered set. Let $\mathbf{H}_{\Gamma} k$ denote the \textbf{Hahn product} of the constant family $(k)_{\gamma \in \Gamma}$, i.e.\ the subspace of $k^\Gamma$ 
        given by $$\mathbf{H}_{\Gamma} k = \{a\in k^\Gamma \colon \tmop{supp} a \text{ is well-ordered}\}.$$
        We also express $a\in \mathbf{H}_{\Gamma}$ as the sum
        $$a = \sum_{\gamma\in \Gamma}a_\gamma \mathds{1}_\gamma,$$
        where $a_\gamma=a(\gamma)$. 
        A natural summability structure $\Sigma$ on $\mathbf{H}_{\Gamma} k$ is obtained as follows. For a set $I$, we define $\tmop{dom} \Sigma_I$ to be the set of maps $\mathbf{v} \colon I \rightarrow \mathbf{H}_{\Gamma} k$ such that the subset $\bigcup \limits_{i \in I} \operatorname{supp} \mathbf{v}(i)$ of $\Gamma$ is well-ordered, and that, for each $\gamma \in \Gamma$, the set $I_{\mathbf{v},\gamma}\assign\{i \in I \ \colon \ \gamma \in \operatorname{supp} \mathbf{v}(i)\}$ is finite. For $\mathbf{v} \in \tmop{dom} \Sigma_I$, we define $\Sigma_I \mathbf{v} \in \mathbf{H}_{\Gamma} k$ to be 

        \[\Sigma_I\mathbf{v} = \sum_{\gamma\in \Gamma}\left(\sum \limits_{i \in I_{\mathbf{v},\gamma}} \mathbf{v}(i)\right)\mathds{1}_\gamma,\]        
        whose support is indeed well-ordered, as it is contained in $\bigcup \limits_{i \in I} \operatorname{supp} \mathbf{v}(i)$. We leave it to the reader to verify that $(\mathbf{H}_{\Gamma} k,\Sigma)$ satisfies the axioms of ultrafinite summability spaces (this will also follow from the more general result \Cref{prop-gloub-spaces}).
    \end{example}

    Let us introduce a natural notion of subspace for summability spaces.
    \begin{definition}\label{def:closedrestriction} We say that a subspace $W\subseteq V$ is {{{\tmstrong{closed}}}} (under $\Sigma$) if for any
            set $I$ and any $\mathbf{w} \in \tmop{dom} \Sigma_I \cap W^I$, we have
            $\Sigma_I \mathbf{w} \in W$.
            
            If $W$ is a closed subspace of $V$, then we define the summability structure $\Sigma^W$ on $W$ by
            $\tmop{dom} \Sigma_I^W \assign \tmop{dom} \Sigma_I \cap W^I \text{ and }\Sigma_I^W \mathbf{w} \assign \Sigma_I \mathbf{w}$
            for any set $I$ and any $\mathbf{w} \in \tmop{dom} \Sigma^W_I$.

    \end{definition}

    We leave it to the reader to check that $(W,
    \Sigma^W)$ in \Cref{def:closedrestriction} is indeed a summability space, and that $(W, \Sigma^W)$ is ultrafinite
    if $(V, \Sigma)$ is ultrafinite.

    \begin{example} If $(\Gamma,<)$ is as in \Cref{ex-Hahn-space} and $\Gamma$ is infinite, then the subspace of $\mathbf{H}_{\gamma  \in \Gamma} k$ of functions $\Gamma \longrightarrow k$ with finite support is not a closed subpsace of $\mathbf{H}_{\gamma  \in \Gamma} k$. Indeed, the family $(\mathds{1}_{\{\gamma\}})_{\gamma \in \Gamma}$ of indicator functions for each singleton $\{\gamma\},\gamma \in \Gamma$ is summable in $\mathbf{H}_{\gamma  \in \Gamma_0} k$ but its sum has infinite support.
 However, given a subset $\Gamma_0$, the subspace $\mathbf{H}_{\gamma  \in \Gamma_0} k \subseteq \mathbf{H}_{\gamma  \in \Gamma} k$ is a closed subspace.

\end{example}

\subsection{Noetherian series}\label{subsection:Noetherianseries}
We now introduce a central example of summability spaces called spaces of Noetherian series. These series generalize the concept of Hahn series and provide the most general setting where we will establish the derivation-automorphism correspondence in \Cref{section-Noetherian-series}.
    
    \begin{definition}
    \label{def-Noetherian}
        Let $(\Omega, <)$ be a partially ordered set. We call $X\subseteq \Omega$ a  \textbf{chain} in $(\Omega, <)$ if it is linearly ordered by $<$. A chain $X$ in $(\Omega, <)$ is
        \textbf{decreasing} if it has no minimal element, i.e.\
        for any $x\in X$ there exists $z\in X$ with $z<x$. We call $Y\subseteq \Omega$ an \textbf{antichain} in $(\Omega, <)$ if any two distinct elements of $Y$ are incomparable, i.e.\ for any $y,z\in Y$ we have $y\not < z$.
        We say that $(\Omega, <)$ is
        {\tmstrong{{\bf{Noetherian}}}}{\index{Noetherian ordering}}, or that $<$
        is a Noetherian ordering on $\Omega$, if any decreasing
        chain in $(\Omega, <)$ is finite and any antichain in $(\Omega, <)$ is finite.
    \end{definition}

    Note that Noetherian linear orderings are exactly well-orderings.

    \begin{proposition}
    \label{prop-Noetherian-subsequence}{\tmem{{\cite[Proposition~A.1]{vdH:phd}}}}
        A partially ordered set $(\Omega, <)$ is Noetherian if and only if every sequence $u\colon \mathbb{N} \longrightarrow \Omega$ has an increasing subsequence.
    \end{proposition}

    \begin{definition}
        Let $(\Omega, <)$ be a partially ordered set. A \textbf{bad sequence} in $(\Omega,<)$ is a sequence $u\colon \mathbb{N}
        \longrightarrow \Omega$ such that for any $i, j \in \mathbb{N}$ with
        $i < j$ we have $u_i \not\leq u_j$. Given a map $f\colon \Omega \longrightarrow
        \mathbb{N}$, a bad sequence $u$ in $\Omega$ is \textbf{minimal} for $f$ if for any bad sequences $v$ in $\Omega$ and any $i \in
        \mathbb{N}$ with $(v_0, \ldots, v_{i -
        1}) = (u_0, \ldots, u_{i - 1})$, we have $f (v_i) \geq f (u_i)$. 
    \end{definition}

\begin{lemma}\label{lem:minbadsequence} Let $(\Omega,<)$ be a partially ordered set and let $f\colon \Omega \longrightarrow \mathbb{N}$ be a function. If there is a bad sequence in $\Omega$, then there is a minimal one for $f$.
\end{lemma}
    
    \begin{proof} We define a minimal bad sequence $(v_n)_{n \in \mathbb{N}}$ by induction (using dependent choice). For each $n \in \mathbb{N}$, we ensure that \begin{itemizedot}
        \item there is a a bad sequence in $\Omega$ which extends $(v_i)_{i <n}$,
        \item each $f(v_i)$ for $i<n$ is among values $f(u_i)$ where $(u_p)_{p \in \mathbb{N}}$ is a bad sequence in $\Omega$.
    \end{itemizedot} 
    If $n$ is such that the conditions hold at $n$, then we pick $v_n$ to be any element of $\Omega$ with minimal value $f(v_n)$ among $(n+1)$-th elements of bad sequences extending $(v_i)_{i <n}$. By construction, any such sequence $(v_i)_{i \in \mathbb{N}}$ is a minimal bad sequence for $f$. 
\end{proof}   
    
    \begin{lemma}
    \label{lem-Noeth-subsequence}{\tmem{{\cite[Theorem~2.1]{Hig52}}}} 
        Let $(\Omega,<)$ be a partially ordered set. Then $(\Omega, <)$ is Noetherian if and only if there is no bad sequence in $\Omega$.
    \end{lemma}

See {\cite[Theorem~2.1]{Hig52}} for other characterisations of Noetherian
orderings.

\begin{proposition}
  \label{prop-Noetherian-gloub}  Let $(\Omega, <)$ be a partially ordered set. Let $A,B$ be Noetherian subsets of $\Omega$. Then $A\cup B$ is Noetherian. Moreover, if $C \subseteq A$, then $C$ is Noetherian.
\end{proposition}

\begin{proof}
  If $C \subseteq A$ are subsets of $\Omega$, then a decreasing chain (respectively
  an antichain) in $C$ is a decreasing chain (respectively an antichain) in $A$. So
  $C$ is Noetherian if $A$ is Noetherian. If $X$ were an infinite decreasing chain in $A \cup B$,
  then $X \cap A$ or $X \cap B$ would be a decreasing chain in $A$ or
  $B$ respectively , which cannot be. If $Y$ is an antichain in $A \cup
  B$, then $Y \cap A$ and $Y \cap B$ are antichains in $A$ and $B$
  respectively, so $Y$ must be finite. So $A \cup B$ is Noetherian.
\end{proof}

    \begin{definition}\label{def:Noetheriansummability}
        Given a non-empty partially ordered set $(\Omega,<)$, we let
    \[\mathbf{H}_{\Omega} k \assign \{{v}\colon \Omega\longrightarrow k\mid \mathrm{supp}\,\mathbf{v}\text{ is Noetherian}\}.\]\label{autolab11}For a set $I$, we define a map $\Sigma^{\Omega}_I$ as follows. The domain $\tmop{dom} \Sigma^{\Omega}_I$ of $\Sigma^{\Omega}_I$ is the set of all $\mathbf{f} \colon I
    \longrightarrow \mathbf{H}_{\Omega} k$ with
    \begin{itemizedot}
        \item $\bigcup_{i \in I} \tmop{supp} \mathbf{f} (i)$ is a Noetherian subset of $\Omega$, and
  
        \item for any $p \in \Omega$, the set $\{i \in I \ \colon \ p \in \mathrm{supp}\,\mathbf{f}(i)\}$ is finite.
    \end{itemizedot}
    For any $\mathbf{f}\in \tmop{dom} \Sigma^{\Omega}_I$,  we define $ \tmop{supp} \Sigma_I^{\Omega}  \mathbf{f}\in \mathbf{H}_{\Omega} k$, whose support satisfies \begin{equation}\label{eq:suppunion}
       \tmop{supp} \Sigma_I^{\Omega} \mathbf{f} \subseteq \bigcup \limits_{i \in I} \tmop{supp} \mathbf{f}(i),
    \end{equation} by
    $$\Sigma_I^{\Omega} \mathbf{f}\colon \Omega\longrightarrow k, p \mapsto \sum_{i \in I} \mathbf{f}(i)
    (p).$$

We say that $(\mathbf{H}_{\Omega} k,\Sigma^{\Omega})$ is the \textbf{space of Noetherian series} (with exponents in ${\Omega}$). 
    \end{definition}

     In the particular case when $(\Omega,<)$ itself is Noetherian, we have
    $\mathbf{H}_{\Omega} k = k^{\Omega}$. In the particular case when $\Omega$ is linearly ordered $\mathbf{H}_{\Omega} k$ coincides with the Hahn product space of \Cref{ex-Hahn-space}.

    \begin{proposition}
    \label{prop-gloub-spaces}Let $(\Omega,<)$ be a non-empty partially ordered set. Then $(\mathbf{H}_{\Omega} k,\Sigma^{\Omega})$ is an ultrafinite summability space.
    \end{proposition}

    \begin{proof}
        Throughout the proof, we fix sets $I, J$ and a generic element $p \in
        \Omega$. We will repeatedly appeal to \Cref{prop-Noetherian-gloub} without explicitly mentioning it. For instance, \Cref{prop-Noetherian-gloub} implies that $\mathbf{H}_{\Omega} k$ is a subspace of $k^{\Omega}$, and that $\Sigma_I^{\Omega}$ ranges in $\mathbf{H}_{\Omega} k$. 
        
        Let $\mathbf{f}, \mathbf{g} \in \tmop{dom} \Sigma^{\Omega}_I$ and let $c \in k$. For all $i \in I$, we
        have
        \[ \tmop{supp} (\mathbf{f} (i) + c\mathbf{g} (i)) \subseteq (\tmop{supp} \mathbf{f} (i)) \cup
        (\tmop{supp} \mathbf{g} (i)) \]
        so $\bigcup_{i \in I} \tmop{supp} ((\mathbf{f} + c\mathbf{g}) (i)) \subseteq \left( \bigcup_{i
        \in I} \tmop{supp} \mathbf{f} (i) \right) \cup \left( \bigcup_{i \in I} \tmop{supp} \mathbf{g}
        (i) \right)$. 

        Since $\{i  \in I \ \colon \ p \in \tmop{supp} \mathbf{f}(i)\}$ and $\{i  \in I \ \colon \ p \in \tmop{supp} \mathbf{g}(i)\}$ are finite, so is $\{i  \in I \ \colon \ p \in \tmop{supp} \mathbf{f}(i)+c \mathbf{g}(i)\}$. This shows that $\mathbf{f} +
        c\mathbf{g}$ lies in $\tmop{dom} \Sigma^{\Omega}_I$. Moreover
        \[ \sum_I (\mathbf{f} (i) (p) + c\mathbf{g} (i) (p))_{i \in I} = \sum_I (\mathbf{f} (i) (p))_{i \in I}
        + c \sum_I (\mathbf{g} (i) (p))_{i \in I}, \]
        so $\sum_{i \in I} (\mathbf{f} + c\mathbf{g}) = \sum_{i \in I} \mathbf{f} + c \sum_{i \in I} \mathbf{g}$.
  
        If $\mathbf{f}$ has finite support, then $\bigcup_{i \in I} \tmop{supp} \mathbf{f} (i)$ is a
        finite union of Noetherian subsets, so it is Noetherian.
        The family $F = (\mathbf{f} (i) (p))_{i \in I}$ has finite support, so $\mathbf{f} \in \tmop{dom} \Sigma^{\Omega}_I$, with $\sum_{i \in I} \mathbf{f} = \sum_{i \in \tmop{supp} F}
        \mathbf{f} (i)$. So \hyperref[ss1]{\textbf{SS1}} holds.
  
        Let $\varphi \colon J \longrightarrow I$ be a bijection. Then we have \[\bigcup_{j
        \in J} \tmop{supp} \mathbf{f} (\varphi (j)) = \bigcup_{i \in I} \tmop{supp} \mathbf{f} (i)\]
        which is Noetherian. The set $\{j \in J \ \colon \ p \in \tmop{supp} \mathbf{f}(\varphi(j))\} = \varphi^{-1}(\{i \in I \ \colon \ p \in \tmop{supp} \mathbf{f}(i)\})$ is finite since $\varphi$ is injective. Thus $(\mathbf{f} (\varphi (j)) (p))_{j \in J} \in \tmop{dom} \Sigma^{\Omega}_J$, and we have
        \[ \left( \sum_{j \in J} \mathbf{f} (\varphi (j)) \right) (p) = \sum_{j \in J} \mathbf{f}
        (\varphi (j)) (p) = \sum_{i \in I} \mathbf{f} (i) (p) = \left( \sum_{i \in I} \mathbf{f}
        (i) \right) (p) . \]
        So $\Sigma^{\Omega}_J \mathbf{f} \circ \varphi = \Sigma^{\Omega}_I \mathbf{f}$, whence \hyperref[ss2]{\textbf{SS2}} holds.
  
        Now assume that $I = \bigsqcup_{j \in J} I_j$. For $j \in J$, we have
        $\bigcup_{i \in I_j} \tmop{supp} \mathbf{f} (i) \subseteq \bigcup_{i \in I}
        \tmop{supp} \mathbf{f} (i)$ so $\bigcup_{i \in I_j} \tmop{supp} \mathbf{f} (i)$ is Noetherian. The set $\{i \in I_j \ \colon \ p \tmop{supp} \mathbf{f}(i)\} \subseteq \{i \in I \ \colon \ p \tmop{supp} \mathbf{f}(i)\}$ is finite. So $(\mathbf{f}
        (i))_{i \in I_j} \in \tmop{dom} \Sigma^{\Omega}_{I_j}$. We set $\mathbf{f}_j \assign \sum_{i \in
        I_j} \mathbf{f} (i)$.
  
        We have
        \[ \bigcup_{j \in J} \tmop{supp} \mathbf{f}_j \subseteq \bigcup_{j \in J} \left(
        \bigcup_{i \in I_j} \tmop{supp} \mathbf{f} (i) \right) \subseteq \bigcup_{i \in I}
        \tmop{supp} \mathbf{f} (i), \]
        so $\bigcup_{j \in J} \tmop{supp} \mathbf{f}_l$ is Noetherian. The set $\{j \in J \ \colon \ p \in \tmop{supp} \mathbf{f}_j\} \subset \{i \in I \ \colon \ p \in \tmop{supp} \mathbf{f}(i)\}$ is finite. So $(\mathbf{f}_j)_{j \in J} \in \tmop{dom} \Sigma^{\Omega}_J$. Moreover, $\sum \limits_{j \in J} \mathbf{f}_j(p) = \sum \limits_{j \in J} \sum \limits_{i \in I_j} \mathbf{f}(i)(p)$. So \hyperref[ss3]{\textbf{SS3}} holds.
  
        Let $I_1, I_2$ be sets with $I = I_1 \sqcup I_2$, and let $\mathbf{f}_1 \in \tmop{dom} \Sigma^{\Omega}_{I_1}$ and $\mathbf{f}_2 \in \tmop{dom} \Sigma^{\Omega}_{I_2}$. We have 

        \[ \bigcup_{i
        \in I} \tmop{supp} (\mathbf{f}_1 \sqcup \mathbf{f}_2) (i) = (\bigcup_{i \in I_1} \tmop{supp}
        \mathbf{f}_1 (i)) \cup (\bigcup_{i \in I_2} \tmop{supp} \mathbf{f}_2 (i)), \] 

        which is Noetherian. The set $\{ i \in I \ \colon \ p \in \tmop{supp} (\mathbf{f}_1 \sqcup \mathbf{f}_2)(i)\}$ is contained in $\{ i \in I \ \colon \ p \in \tmop{supp} \mathbf{f}_1(i)\} \cup \{ i \in I \ \colon \ p \in \tmop{supp} \mathbf{f}_2(i)\}$, so it is finite. Therefore $\mathbf{f}_1
        \sqcup \mathbf{f}_2 \in \tmop{dom} \Sigma^{\Omega}_I$. So \hyperref[ss4]{\textbf{SS4}} holds.
  
        Let $(h_i)_{i \in I}$ be a
        family indexed by $I$ of $k$-valued maps $h_i$ with finite domains
        $X_i$. We have 

        \[ \bigcup_{i \in I} \bigcup_{x \in X_i} \tmop{supp} h_i (x) \mathbf{f}
        (i) \subseteq \bigcup_{i \in I} \tmop{supp} \mathbf{f} (i),\] 
        which is Noetherian. The set $\{(i,x) \ \colon \ i \in I,x \in X_i,p\in \tmop{supp} h_i(x)\mathbf{f}(i)(p)\}$ is contained in $\bigcup \limits_{i \in Y} \{i\} \times X_i$ where $Y$ is the finite set $Y\assign \{i \in I \ \colon \ p \in \tmop{supp} \mathbf{f}(i)(p)\}$. Thus this set is finite, i.e.  $(h_i (x) \mathbf{f} (i))_{i \in I \wedge x \in
        X_i} \in \tmop{dom} \Sigma^{\Omega}_{I \mathbf{f}}$. So \hyperref[uss]{\textbf{UF}} holds.
    \end{proof}

We will always consider $\mathbf{H}_{\Omega} k$ as endowed with the summability structure of \Cref{def:Noetheriansummability}.

    \begin{example} In the case when $(\Omega,<)$ is linearly ordered, the space $\mathbf{H}_{\Omega} k$ is the the Hahn product of \Cref{ex-Hahn-space}.\end{example}

\subsection{Strongly linear maps}\label{subsection-strongly-linear-maps}

   For the sequel of \Cref{subsection-strongly-linear-maps}, we fix two summability spaces $(V, \Sigma)$ and $(W,\Sigma')$. 

    \begin{definition}
        A map $\phi\colon V \longrightarrow W$ is called
        {\bf{{\tmstrong{strongly linear}}}} (with respect to $\Sigma$ and $\Sigma'$) if
        it is linear, and if, for all sets $I$ and families $\mathbf{v} \in \tmop{dom}
        \Sigma_I$, the family $\phi \circ \mathbf{v}$ lies in $\tmop{dom} \Sigma'_I$
        and
        \[ \Sigma'_I  (\phi \circ \mathbf{v}) = \phi (\Sigma_I \mathbf{v}) . \]
        We write $\tmop{Lin}^+ (V, W)$\label{autolab12} for the set of strongly linear
        maps $V \longrightarrow W$, and $\tmop{Lin}^+ (V) \assign
        \tmop{Lin}^+ (V, V)$\label{autolab13} for the set of strongly linear endomorphisms on $V$.
    \end{definition}

    \begin{example}
        If $\Sigma$ is the minimal summability structure on $V$, then we have
        $\tmop{Lin}^+ (V, W) = \tmop{Lin} (V, W)$, i.e. every linear map is strongly linear with respect to the minimal summability structure on $V$.
    \end{example}

    \begin{example}
        Consider the Hahn space $V=\mathbf{H}_{\gamma \in \Gamma} k$ of \Cref{ex-Hahn-space}. For all order-preserving maps $u\colon\Gamma \longrightarrow \Gamma$, the map $V \longrightarrow V\ \colon \ f \mapsto f \circ u$ is strongly linear (see \cite[Proposition 2.2.4]{Serra:phd} for a proof).
    \end{example}

    Linear maps between summability spaces need not be strongly linear, but counterexamples are typically produced using non-constructive methods (see \cite[Example 4.0.8]{KuhlSer:vaut}).

    \begin{example}
        If $V=\mathbf{H}_{\gamma \in \Gamma} k$ is the Hahn space of \Cref{ex-Hahn-space}, then given a strongly linear map $\mu\colon V\longrightarrow V$ and $a \in V$, we have \[\mu(a) = \mu(\sum \limits_{\gamma \in \Gamma} a(\gamma) \mathbbm{1}_{\gamma})=\sum \limits_{\gamma \in \Gamma} a(\gamma) \mu(\mathbbm{1}_{\gamma}).\]
        In particular, the map $\mu$ is determined by the family $(\mu(\mathbbm{1}_{\gamma}))_{\gamma \in \Gamma}.$
    \end{example}

    We now investigate the algebraic structure of spaces
of strongly linear maps between two summability spaces.

    \begin{proposition}
        \label{prop-Lin+-subspace}$\tmop{Lin}^+ (V, W)$ is a subspace of
        $\tmop{Lin} (V, W)$.
    \end{proposition}

    \begin{proof}
        Let $\phi, \psi \in \tmop{Lin}^+ (V, W)$, $c \in k$, let $I$ be a set
        and let $\mathbf{v} \in \tmop{dom} \Sigma_I$. Since $\tmop{dom} \Sigma'_I$
        is a vector subspace of $W^I$ and $\Sigma'_I$ is a linear map, the
        family $\phi \circ \mathbf{v} + c \psi \circ \mathbf{v}$ is summable with sum
        \[ \Sigma'_I  (\phi \circ \mathbf{v}) + c \Sigma'_I  (\psi \circ \mathbf{v}) = \phi
        (\Sigma_I \mathbf{v}) + c \psi (\Sigma_I \mathbf{v}) = (\phi + c \psi)
        (\Sigma_I \mathbf{v}) . \]
        This shows that $\phi + c \psi \in \tmop{Lin}^+ (V, W)$.
    \end{proof}

    \begin{proposition}
    \label{prop-lin-subalgebra}The composition of two strongly linear maps $V\rightarrow V$ is again strongly linear. Therefore $(\tmop{Lin}^+ (V),+,\circ)$ is a subalgebra of
        $(\tmop{Lin} (V), +, \circ)$.
    \end{proposition}

    \begin{proof}
        Let $\phi, \psi \in \tmop{Lin}^+ (V)$, let $I$ be a set and let $\mathbf{v} \in
        \tmop{dom} \Sigma_I$. Then $\psi \circ \mathbf{v}$ is summable with sum
        $\sum_I \psi \circ \mathbf{v} = \psi \left( \sum_I \mathbf{v} \right)$ by strong
        linearity of $\psi$. So $(\phi \circ (\psi \circ \mathbf{v}))_{i \in I}$ is summable
        by strong linearity of $\phi$, with sum $\sum_I \phi \circ (\psi \circ
        \mathbf{v}) = \phi \left( \psi \left( \sum_I \mathbf{v} \right) \right)$. Since $\phi
        \circ \psi$ is linear, this shows that $\phi \circ \psi \in \tmop{Lin}^+
        (V)$.
    \end{proof}

    \begin{corollary}
        $(\tmop{Lin}^+ (V),+,\circ)$ is a Lie subalgebra of $\left( \tmop{Lin} (V), +,
        \left[ \hspace{1.0em} \right], . \right)$.
    \end{corollary}
    
        We next equip $\tmop{Lin}^+ (V, W)$ with a summability structure  (see \Cref{prop-multi-strong} below). 

        \begin{definition}
        Let $J$ be a set, and let $\boldsymbol{\phi}\colon J \longrightarrow \tmop{Lin}^+ (V,
        W)$ be a map. 

\begin{enumerate}[(i)]

\item We say that $\boldsymbol{\phi}$ is {$\tmop{Lin}$\textbf{-summable}}{\index{$\tmop{Lin}$-summable family}} if for all sets $I$ and all $\mathbf{v} \in \tmop{dom}
\Sigma_I$, the family $\boldsymbol{\phi} (\mathbf{v}) \assign (\boldsymbol{\phi} (j) (\mathbf{v} (i)))_{(i, j) \in I
\times J}$ lies in $\tmop{dom}
\Sigma'_{I \times J}$.

\item If $\boldsymbol{\phi}$ is $\tmop{Lin}$-summable, then we define a map $\sum_{j
\in J} \boldsymbol{\phi} (j) \colon V \longrightarrow W$ as follows. For $v \in V$, define
\[ \left( \sum_{j \in J} \boldsymbol{\phi} (j) \right) (v) \assign \Sigma'_J (\boldsymbol{\phi}(j)(v))_{j \in J}. \]

   \end{enumerate}

   \end{definition}

    \begin{lemma}
        \label{lem-end-sum-stronglin}If $J$ is a set and $\boldsymbol{\phi} \colon J \longrightarrow
        \tmop{Lin}^+ (V, W)$ is $\tmop{Lin}$-summable, then its sum $\psi \assign
        \sum_{j \in J} \boldsymbol{\phi} (j)$ is strongly linear.
    \end{lemma}

    \begin{proof}
        Let $u_0,v_0 \in V$ and $c \in k$. The linearity of each $\boldsymbol{\phi} (j)$ for $j \in J$ gives that $\boldsymbol{\phi} (u_0 + cv_0)
        = \boldsymbol{\phi} (u_0) + c \boldsymbol{\phi} (v_0)$. Now $\Sigma_{2, J}$ is linear, so $\phi (u_0
        + cv_0) = \psi (u_0) + c \psi (v_0)$, i.e.\ $\psi$ is linear. 
  
        We next prove that $\psi$ is strongly linear. Let $I$ be a set, let $\mathbf{v} \in
        \tmop{dom} \Sigma_I$ and set
        \[ v \assign \Sigma_I \mathbf{v} \in V . \]
        The family $\boldsymbol{\phi} (\mathbf{v})$ is summable. By \hyperref[ss3]{\textbf{SS3}}, both families
        $\left( \sum_{j \in J} \boldsymbol{\phi}(j) (\mathbf{v} (i)) \right)_{i \in I}$ and
        $\left( \sum_{i \in I} \boldsymbol{\phi}(j) (\mathbf{v} (i)) \right)_{j \in J}$ are summable with
        \[ \Sigma'_J (\Sigma'_I (\boldsymbol{\phi} (j) \circ \mathbf{v}))_{j \in J} = \Sigma'_I  (\Sigma'_J \boldsymbol{\phi} (\mathbf{v} (i)))_{i \in I} = \Sigma'_I  (\psi \circ
        \mathbf{v}) . \]
        By strong linearity of each $\boldsymbol{\phi}(j)$, we have $\Sigma'_I (\boldsymbol{\phi} (j) \circ
        \mathbf{v}) = \boldsymbol{\phi} (j) (v)$ for all $j \in J$, whence $\phi (v) = \Sigma'_I
        (\psi \circ \mathbf{v})$. This shows that $\psi$ is strongly linear. 
    \end{proof}

    For each set $I$, we define a map $\Sigma^{\tmop{Lin}}_I$ as follows. The domain of $\Sigma^{\tmop{Lin}}_I$ is the set of $\tmop{Lin}$-summable families $I \rightarrow \tmop{Lin}^+(V,W)$. For such a family $\boldsymbol{\phi}$, we define $\Sigma^{\tmop{Lin}}_I=\Sigma_{i \in I} \boldsymbol{\phi}(i) \in \tmop{Lin}^+ (V, W)$.
    \begin{proposition}
    \label{prop-multi-strong}
        The structure $(\tmop{Lin}^+ (V,W),\Sigma^{\tmop{Lin}})$ is a summability space. Moreover, if $(W, \Sigma')$ is ultrafinite, then $(\tmop{Lin}^+ (V, W),\Sigma^{\tmop{Lin}})$ is ultrafinite.
    \end{proposition}

    \begin{proof}
        Let $J$ be a set, let $(\boldsymbol{\phi}(j))_{j \in J}$ and $(\boldsymbol{\psi}(j))_{j \in J}$ be
        $\tmop{Lin}$-summable, and let $c \in k$. We also fix once and for all a set $I$ and a summable family $\mathbf{v} \in \dom \Sigma_I$. 
  
        The family \[(\boldsymbol{\phi}(j) (\mathbf{v}
        (i)) + c \boldsymbol{\psi}(j) (\mathbf{v} (i)))_{(i, j) \in I \times J}\] is summable by \hyperref[ss2]{\textbf{SS2}}
        in $(W, \Sigma')$. So $\boldsymbol{\phi} + c \boldsymbol{\psi}$ is $\tmop{Lin}$-summable. The axiom \hyperref[ss1]{\textbf{SS1}} follows
        trivially from the validity of \hyperref[ss1]{\textbf{SS1}} in $(W, \Sigma')$, and likewise
        \hyperref[ss4]{\textbf{SS4}} follows from the validity of \hyperref[ss4]{\textbf{SS4}} in $(W, \Sigma')$.
  
        If $J_1, J_2$ are sets and $\psi \colon J_1 \longrightarrow J_2$ is a
        bijection, then for all $\tmop{Lin}$-summable $\boldsymbol{\phi}:J_2\longrightarrow \tmop{Lin}^+(V,W)$, the map
        \begin{eqnarray*}
            (\psi, \tmop{Id}) \colon J_1 \times I & \longrightarrow & J_2 \times I\\
            (j, i) & \longmapsto & (\psi(j), i)
        \end{eqnarray*}
        is bijective, so by \hyperref[ss2]{\textbf{SS2}} in $(W, \Sigma')$, the family
        $(\boldsymbol{\phi}(\psi (j)) (\mathbf{v} (i)))_{(j, i) \in J_1 \times I}$ is summable. Therefore
        $\boldsymbol{\phi} \circ \psi$ is $\tmop{Lin}$-summable. We easily see that\hyperref[ss2]{\textbf{SS2}} holds.
  
        Let us now show that \hyperref[ss3]{\textbf{SS3}} holds. Suppose that $J = \bigsqcup_{l \in L} J_l$ for a set $L$. For $l \in
        L$, we write $\boldsymbol{\phi}_l = \phi \upharpoonleft J_l$. Note that
        \[ I \times J = \bigsqcup_{l \in L} I \times J_l . \]
        It follows that and from \hyperref[ss3]{\textbf{SS3}} in $(W, \Sigma')$ that each family
        $\boldsymbol{\phi}_l (\mathbf{v})$ for $l \in L$ is summable in $(W, \Sigma')$. So
        each $\boldsymbol{\phi}_l$ is $\tmop{Lin}$-summable. Write $\boldsymbol{\sigma}(l) = \sum_{j \in J_l} \phi_j$
        for each $l \in L$. We claim that $\boldsymbol{\sigma} \assign (\boldsymbol{\sigma}(l))_{l \in L}$ is
        $\tmop{Lin}$-summable. Indeed, the family
        \[ \boldsymbol{\sigma} (\mathbf{v}) = \left( \sum_{j \in J_l} \phi_j (\mathbf{v} (i)) \right)_{(i, l) \in I
        \times L} \]
        is summable by \hyperref[ss3]{\textbf{SS3}} in $(W, \Sigma')$. We deduce that \hyperref[ss3]{\textbf{SS3}} holds.
  
        Let us next prove that \hyperref[ss4]{\textbf{SS4}} holds. Suppose that $J = J_1 \sqcup J_2$, that
        $\boldsymbol{\phi}_1 \colon J_1 \longrightarrow \tmop{Lin}^+ (V, W)$ and $\boldsymbol{\phi}_2 \colon J_2
        \longrightarrow \tmop{Lin}^+ (V, W)$ are $\tmop{Lin}$-summable and write $\boldsymbol{\phi}_1
        \sqcup \boldsymbol{\phi}_2 = \boldsymbol{\varphi}$. Then we have $\boldsymbol{\varphi} (\mathbf{v}) = \boldsymbol{\phi}_1 (\mathbf{v}) \sqcup \boldsymbol{\phi}_2
        (\mathbf{v})$, so $\boldsymbol{\varphi} (\mathbf{v})$ is summable by \hyperref[ss4]{\textbf{SS4}} in $(W, \Sigma')$.
  
        Assume now that $(W, \Sigma')$ is ultrafinite. Let $(f
        (j))_{j \in J}$ be a family of $k$-valued maps with finite domains $\dom f_j= X_j$ for each $j \in J$. The family $(f_j(x) \boldsymbol{\phi} (j) (\mathbf{v} (i)))_{(i, j) \in I \times J \wedge x \in X_j}$
        is summable by ultrafiniteness of $(W, \Sigma')$, so $f \boldsymbol{\phi}$ is
        $\tmop{Lin}$-summable. Thus \hyperref[uss]{\textbf{UF}} holds.
    \end{proof}

    he following proposition provides a criterion for
Lin-summability. 

    \begin{proposition}
    \label{prop-multisummability-criterion}Assume that $(V, \Sigma)$ is
        ultrafinite. Let
        $(\Omega,<)$ be a non-empty partially ordered set. Let $J$
        be a set and let $\boldsymbol{\phi} \colon J\longrightarrow \tmop{Lin}^+(
        \mathbf{H}_{\Omega} k, V)$ be a map. Then $\boldsymbol{\phi}$
        is $\tmop{Lin}$-summable if and only if for all Noetherian subsets $S \subseteq \Omega$, the family
        $(\boldsymbol{\phi}(j) (\mathbbm{1}_{p}))_{(j, p) \in J \times S}$ is summable in $(V,
        \Sigma)$.
    \end{proposition}

    \begin{proof}
        We need only prove the ``if'' direction of the equivalence. Let $I$ be a set and let $\mathbf{f}
        \colon I \longrightarrow \mathbf{H}_{\Omega} k$ be summable in $(\mathbf{H}_{\Omega} k,\Sigma^{\Omega})$.
        Write $I_p \assign \{ i \in I \suchthat p \in \tmop{supp} \mathbf{f} (i) \}$ for each $p \in \Omega$, and set
        \[ S \assign \bigcup_{i \in I} \tmop{supp} \mathbf{f} (i). \]
        Since $S$ is Noetherian, the family $(\boldsymbol{\phi}(j) (\mathbbm{1}_{\{ p \}}))_{j \in J
        \wedge p \in S}$ is summable in $(V, \Sigma)$. For $p \in \Omega$, the set $I_p$ is finite. We
        deduce by ultrafiniteness that the family $(\boldsymbol{\phi}(j) (\mathbbm{1}_p \mathbf{f}(i)
        (p)))_{j \in J \wedge p \in \Omega \wedge i \in I_p}$ is summable, whence
        $(\boldsymbol{\phi}(j) (\mathbbm{1}_p \mathbf{f}(i) (p)))_{j \in J \wedge p \in \Omega \wedge i \in
        I}$ is summable by \hyperref[ss4]{\textbf{SS4}} in $(V, \Sigma)$. It follows by \hyperref[ss3b]{\textbf{SS3b}} in $(V,\Sigma)$ that the family $\left( \sum_{p \in S} \boldsymbol{\phi}(j)
        (\mathbbm{1}_p \mathbf{f}(i) (p)) \right)_{j \in J \wedge i \in I}$ is summable.
        Each $\boldsymbol{\phi}(j)$ is strongly linear, so $(\boldsymbol{\phi}(j) (\mathbf{f} (i)))_{j \in J
        \wedge i \in I} = \left( \sum_{p \in S} \boldsymbol{\phi}(j) (\mathbbm{1}_p \mathbf{f}(i) (p))
        \right)_{j \in J \wedge i \in I}$. So $(\boldsymbol{\phi}(j) (\mathbf{f} (i)))_{j \in J \wedge i
        \in I}$ is summable, i.e.\ $\boldsymbol{\phi}$ is $\tmop{Lin}$-summable.
    \end{proof}

\subsection{Summability algebras}

    Let $(A, +, \cdot)$ be a algebra, where $(A, +, \cdot)$ is a possibly
    non-commutative, possibly non-unital ring. Let $\Sigma$ be a summability structure
    on the underlying vector space of $A$.

Given two sets $I,J$ and two maps $\mathbf{a} \colon I \longrightarrow A$ and $\mathbf{b} \colon J \longrightarrow A$, we write $\mathbf{a} \cdot \mathbf{b}$ for the map $I \times J \longrightarrow A \ ; (i,j) \mapsto \mathbf{a}(i) \cdot \mathbf{b}(j)$.
    \begin{definition}\label{def:summabilityalgebra}
        We say that $(A,+,\cdot, \Sigma)$ is a {\bf{summability
        algebra}}{\index{summability algebra}} if the following axiom is satisfied.

        \begin{descriptioncompact}
            \item[$\mathbf{SA}$] \label{sa}For any sets $I,J$ and any $(\mathbf{a}, \mathbf{b}) \in \tmop{dom} \Sigma_I
            \times \tmop{dom} \Sigma_J$, we have $\mathbf{a} \cdot \mathbf{b} \in \tmop{dom} \Sigma_{I \times J}$, and
            \[ \sum_{I \times J} (\mathbf{a} \cdot \mathbf{b}) = \left( \sum_I \mathbf{a} \right) \cdot \left(
                \sum_J \mathbf{b} \right) . \]
        \end{descriptioncompact}
    \end{definition}

    We write $\tmop{End}
  (A)$\label{autolab15} for the set of endomorphisms of algebra of $A$, and $\tmop{End}^+
  (A) \allowbreak = \tmop{End}(A) \cap \tmop{Lin}^+(A)$ for the set of strongly linear endomorphisms of algebra of $A$. By \Cref{prop-lin-subalgebra}, the set $\tmop{End}^+ (A)$ is closed under
composition.

    \begin{proposition}
    \label{prop-strong-products}Let $(A,+,\cdot, \Sigma)$ be a summability algebra. Then for any
        $a \in A$, the left and right product maps $\mathord{a \ \cdot} \: \colon A
        \longrightarrow A ; b \mapsto a \cdot b$ and $\mathord{\: \cdot \ a} \: \colon A
        \longrightarrow A ; b \mapsto b \cdot a$ are strongly linear.
    \end{proposition}

    \begin{proof}
        By \hyperref[sa]{\textbf{SA}}, for all summable families $(b_i)_{i \in I}$ in $A$, the
        families $(a \cdot b_i)_{i \in I}$ and $(b_i \cdot a)_{i \in I}$ are
        summable, with $\sum_{i \in I} a \cdot b_i = a \cdot \left( \sum_{i \in I}
        b_i \right)$ and $\sum_{i \in I} b_i \cdot a = \left( \sum_{i \in I} b_i
        \right) \cdot a$.
    \end{proof}

    We next introduce our two main examples of summability algebras. Fix a partially ordered monoid $(M, +, 0, <)$, i.e.\ a monoid $(M, +, 0)$ together
with a partial ordering $<$ on $M$ with
\begin{equation}
  \forall f, g, h \in M, f < g \Longrightarrow (f + h < g + h \wedge h + f < h
  + g) . \label{eq-ordered-monoid}
\end{equation}

\begin{lemma}\label{lem:noethsum} \tmem{\cite[(2.1)]{ribenboim}} Let $X,Y \subseteq M$ be Noetherian subsets. Then $X+Y\assign \{x+y \ \colon \ (x,y) \in X \times Y\}$ is Noetherian.
\end{lemma}

\begin{lemma}\label{lem:noethfin} \tmem{\cite[(4.1)]{ribenboim}} Let $X,Y \subseteq M$ be Noetherian subsets. For $m \in M$, the set $I\assign \{ (x,y) \in X \times Y \suchthat
     x + y = m \}$ is finite.
\end{lemma}

We
write $k(\!(M)\!) \assign \mathbf{H}_M k$, and we regard it as an ultrafinite summability space for the summmability structure $\Sigma$ of \Cref{prop-gloub-spaces}. As in {\cite{Serra:phd,KuhlSer:vaut}}, given
$m \in M$, we write $t^m \assign \mathbbm{1}_m$.
For each $a \in k(\!(M)\!)$, the
family $(a (m) t^m)_{m \in M}$ is summable, with sum
\[ a = \sum_{m \in M} a (m) t^m . \]
The vector space $k(\!(M)\!)$ is equipped \cite[Section~4]{ribenboim} with the Cauchy product
\begin{equation}
  \forall m \in M, (a \cdot b) (m) \assign \sum_{m_1 + m_2 = m} a (m_1) b (m_2),
  \label{eq-Noetherian-prod}
\end{equation} for which it is a unital ring with
  multiplicative identity $1 \assign t^0$. It is easy to see that it is an algebra over $k$.

\begin{proposition}\label{prop-Noeth-salgebra}
  The structure $(k(\!(M)\!),\Sigma)$ is a summability
  algebra.
\end{proposition}

\begin{proof}
  We have to show that \hyperref[sa]{\textbf{SA}} holds. Let $I,J$ be sets, let $\mathbf{a}\colon I \longrightarrow k(\!(M)\!)$ and $\mathbf{b}\colon J \longrightarrow k(\!(M)\!)$ be summable. Let $m \in M$. We have

  \begin{eqnarray*}
  \bigcup \limits_{(i,j) \in I \times J} \tmop{supp} \mathbf{a}(i) \cdot \mathbf{b}(j) & \subseteq & \bigcup \limits_{(i,j) \in I \times J} \tmop{supp} \mathbf{a}(i) + \tmop{supp} \mathbf{b}(j) \\
  & \subseteq & \bigcup \limits_{i \in I} \tmop{supp} \mathbf{a}(i) +\bigcup \limits_{j \in J} \tmop{supp} \mathbf{b}(j),\\
\end{eqnarray*}
which is Noetherian by \Cref{lem:noethsum}. Let $m \in M$, and write $S \assign \{(m_1,m_2) \in \bigcup \limits_{i \in I} \tmop{supp} \mathbf{a}(i) \times \bigcup \limits_{j \in J} \tmop{supp} \mathbf{b}(j) \suchthat m= m_1+m_2\}$. Note that $I$ is finite by \Cref{lem:noethfin}. Now the set $\{(i,j) \in I \times J \ \colon \ m \in \tmop{supp} \mathbf{a}(i)\cdot \mathbf{b}(j)\}$ is contained in the finite union $\bigcup \limits_{(m_1,m_2) \in S} \{i \in I \suchthat m_1 \in \tmop{supp} \mathbf{a}(i)\} \times \{j \in J \suchthat m_2 \in \tmop{supp} \mathbf{b}(j)\}$ of cartesian products of finite sets, so it is finite. Thus $\mathbf{a} \cdot \mathbf{b}$ is summable.
  
  For $m \in M$, the previous arguments give
  \begin{eqnarray*}
    \left( \sum_{(i, j) \in I \times J} \mathbf{a}(i) \cdot \mathbf{b}(j) \right) (m) & = &
    \sum_{m_1 +m_2 = m} \sum_{(i, j) \in I \times
    J} \mathbf{a}(i) (m_1) \mathbf{b}(j) (m_2)\\
    & = & \sum_{m_1+m_2=m} \left( \sum_{i \in I} \mathbf{a}(i)
    (m_1) \right)  \left( \sum_{j \in J} \mathbf{b}(j) (m_2) \right)\\
    & = & \sum_{m_1+m_2 = m} \left( \sum_{i \in I} \mathbf{a}(i) \right)
    (m_1)  \left( \sum_{j \in J} \mathbf{b}(j) \right) (m_2)\\
    & = & \left( \left( \sum_{i \in I} \mathbf{a}(i) \right) \cdot \left( \sum_{j \in
    J} \mathbf{b}(j) \right) \right) (m) .
  \end{eqnarray*}
  This shows that $\sum_{(i, j) \in I \times J} \mathbf{a}(i) \cdot \mathbf{b}(j) = ( \sum_{i
  \in I} \mathbf{a}(i))  \cdot ( \sum_{j \in J} \mathbf{b}(j))$, hence that
  \hyperref[sa]{\textbf{SA}} holds.
\end{proof}

In particular, for $a, b \in k(\!(M)\!)$, the family $(a (m_1) b (m_2) t^{m_1 +
m_2})_{m_1, m_2 \in M}$ is summable, with sum
\[ a \cdot b = \sum_{m_1, m_2 \in M} a (m_1) b (m_2) t^{m_1 + m_2} . \]
We call $k(\!(M)\!)$ the {\bf{algebra of Noetherian
series}}{\index{algebra of Noetherian series}} with exponents in $M$.

Recall from \Cref{subsection-strongly-linear-maps} that given a summability space $(V,\Sigma)$, the space of strongly linear maps $V\rightarrow V$ is a summability space for a natural structure $\Sigma^{\tmop{Lin}}$, and that it is a subalgebra of $\tmop{Lin}(V)$.

\begin{proposition}\label{prop:Linalgebra}
        Let $(V, \Sigma)$ be a summability space. Then $(\tmop{Lin}^+ (V),+,\circ,\Sigma^{\tmop{Lin}})$ is a summability algebra.
    \end{proposition}

    \begin{proof}
        By \Cref{prop-multi-strong}, we need only prove that \hyperref[sa]{\textbf{SA}} holds in
        $\tmop{Lin}^+ (V)$. Let $I$ and $J$ be sets, let $\boldsymbol{\phi} \colon I\longrightarrow \tmop{Lin}^+(V)$ and $\boldsymbol{\psi} \colon J \longrightarrow \tmop{Lin}^+(V)$ be
        $\tmop{Lin}$-summable with respective sums $\phi$ and $\psi$. Let $L$ be a set and let $\mathbf{v}
        \colon L \longrightarrow V$ be summable in $(V, \Sigma)$. 
  
        Since $\boldsymbol{\psi}$ is $\tmop{Lin}$-summable, the family $(\boldsymbol{\psi}(j) (\mathbf{v} (l)))_{(j, l) \in J \times
        L}$ is summable in $(V, \Sigma)$. Since $\boldsymbol{\phi}$ is
        $\tmop{Lin}$-summable, the family $((\boldsymbol{\phi}(i) \circ \boldsymbol{\psi}(j)) (\mathbf{v} (l)))_{(i, j, l) \in I
        \times J \times L}$ is summable in $(V, \Sigma)$. This shows that $(\boldsymbol{\phi}(i)
        \circ \boldsymbol{\psi}(j))_{(i, j) \in I \times J}$ is $\tmop{Lin}$-summable. Let $v_0 \in V$. As
        above, the family
        \[ F \assign ((\boldsymbol{\phi}(i) \circ \boldsymbol{\psi}(j)) (v_0))_{(i, j, l) \in I \times J} \]
        is summable in $(V, \Sigma)$. We have
        \begin{align*}
            (\phi \circ \psi) (v_0) & = \left( \sum_{i \in I} \boldsymbol{\phi}(i) \right) \left( \sum_{j
            \in J} \boldsymbol{\psi}(j) (v_0) \right)&& \\
            & = \sum_{i \in I} \boldsymbol{\phi}(i) \left( \sum_{j \in J} \boldsymbol{\psi}(j) (v_0) \right)&& \\
            & = \sum_{i \in I} \sum_{j \in J} \boldsymbol{\phi}(i) (\boldsymbol{\psi}(j) (v_0))
            && \text{(each $\boldsymbol{\phi}(i)$ is strongly linear)}\\
            & = \sum_{(i, j) \in I \times J} \boldsymbol{\phi}(i) (\boldsymbol{\psi}(j) (v_0)) .
            && \text{(\hyperref[ss3]{\textbf{SS3}})}
        \end{align*}
        So $\phi \circ \psi = \sum_{I \times J} \boldsymbol{\phi}(i) \circ \boldsymbol{\psi}(j)$. Therefore \hyperref[sa]{\textbf{SA}} holds
        in $\tmop{Lin}^+ (V)$.
    \end{proof}

    A {\bf{closed ideal}}{\index{closed ideal}} of a summability algebra $(A,+,\cdot,
    \Sigma)$ is a {\emph{two-sided}} ideal of $A$ which is also a closed subspace
    of~$(A, \Sigma)$.

\subsection{Algebras of formal power series}\label{subsection-kJJ}
    
    \subsubsection*{The free monoid $J^{\star}$ of words over an alphabet $J$} Let $J$ be a set. Write\label{autolab17}
\[ J^{\star} \assign \bigcup_{n \in \mathbb{N}} J^n \]
where $J^0 = \{ \emptyset \}$ and $\emptyset$ denotes the empty word. We see elements of $J^{\star}$ as
{\bf{finite words}}{\index{word}} with alphabet $J$.

For $m, n \in \mathbb{N}$, if $\beta = (\beta_1, \ldots, \beta_m), \gamma = 
(\gamma_1, \ldots, \gamma_n) \in J^{\star}$, we define\label{autolab18}
\[ \beta \gamma \assign (\beta_1, \ldots, \beta_m, \gamma_1, \ldots,
   \gamma_n) \in J^{m + n}\subseteq J^{\star}, \]
where it is implied that $\emptyset \theta = \theta  \emptyset =
\theta$ for all $\theta \in J^{\star}$. This concatenation operation endows $J^{\star}$ with a structure of
cancellative monoid such that for $\theta \in
J^{\star}$, the set
\begin{equation}
  \{ (\beta, \gamma) \in J^{\star} \times J^{\star} \suchthat \theta = \beta
  \gamma \} \label{eq-set-finite}
\end{equation}
is finite (it has exactly $n + 1$ elements when $\theta \in J^n$).

The following 
construction, also considered in \cite[Chapter~0, p 17]{reutenauer}, is a particular case of Bourbaki's
notion of total algebra {\cite[Chapter~III, Section~2.10]{Bou:alg}}. 

\begin{proposition}
  The vector space $k^{J^{\star}}$ is a unital algebra under the Cauchy product
  \begin{equation}
    (P \cdot Q) (\theta) \assign \sum_{\theta =
    \beta \gamma} P (\beta) Q (\gamma) , \label{eq-def-prod}
  \end{equation}
  for any $P,Q\in k^{J^{\star}}$ and any $\theta \in J^{\star}$.
\end{proposition}

    \begin{notation} \label{notation-kJ}
        We write $k \langle\! \langle J \rangle\! \rangle \assign
k^{J^{\star}}$\label{autolab19}.
    \end{notation}

    We recall some properties of $k \langle\! \langle J \rangle\! \rangle$.
For $P \in k \langle\! \langle J \rangle\! \rangle$ and $n \in \mathbb{N}$, we
write
\[ \tmop{supp}_n P \assign (\tmop{supp} P) \cap J^n = \{ \theta \in J^n
   \suchthat P (\theta) \neq \emptyset \} . \]
Note that we have
\begin{equation}
  \tmop{supp}_n P \cdot Q \subseteq \bigcup_{m + p = n} (\tmop{supp}_m P) \
  (\tmop{supp}_p Q) \label{eq-support-identity},
\end{equation}
where $A B = \{ a b \suchthat (a, b) \in A \times B \}$ for all subsets
$A, B \subseteq J^{\star}$. Concatenation of subsets of $J^{\star}$ is
associative. An easy induction gives:

\begin{lemma}
  \label{lem-support-multiproduct}For $m, n \in \mathbb{N}$ and $P_1, \ldots,
  P_n \in k \langle\! \langle J \rangle\! \rangle$, we have
  \[ \tmop{supp}_m P_1 \cdots P_n \subseteq \bigcup_{m_1 + \cdots + m_n = m}
     (\tmop{supp}_{m_1} P_1) \cdots (\tmop{supp}_{m_n} P_n) . \]\qed
\end{lemma}

The set
\[ k \langle\! \langle J \rangle\! \rangle_- \assign \{ P \in k \langle\! \langle J
   \rangle\! \rangle \suchthat \emptyset \nin \tmop{supp} P \} = \{P - P
   (\emptyset) \suchthat P \in k \langle\! \langle J \rangle\! \rangle
   \} \]
is a two-sided ideal in $k \langle\! \langle J \rangle\! \rangle$.

Given $\theta \in J^{\star}$, we write $X_{\theta}$ for the map $\mathbbm{1}_{\{\theta\}} \colon J^{\star}
\longrightarrow k$ with support $\{ \theta \}$ and $X_{\theta} (\theta) = 1$.
So $X_{\emptyset} = 1$, and writing $\theta = (\theta_1, \ldots, \theta_n)$,
we have
\[ X_{\theta} = X_{\theta_1} \cdots X_{\theta_n} . \]

We will write $k \langle\! \langle m \rangle\! \rangle \assign k \langle\! \langle J
\rangle\! \rangle$\label{autolab20} for $m \in \mathbb{N}$ and $J = \{ 0, \ldots, m - 1 \}$. Note that
$k \langle\! \langle 0 \rangle\! \rangle = k$ and that $k \langle\! \langle 1
\rangle\! \rangle$ is the commutative algebra of power series in one variable
$X_0$ and with coefficients in $k$.
\\\\ We now define a summability structure on $k \langle\! \langle J \rangle \!
\rangle$. Let $<$ be a well-ordering on $J$. We extend $<$ into the lexicographic ordering $<_l$ on $J^{\star}$, and note that $(J^{\star},<_l)$ is an ordered monoid under concatenation. Note furthermore that $(J^{\star},<_l)$ is well-ordered, so each subset of $J^{\star}$ is Noetherian. This means that $k \langle\! \langle J \rangle \!
\rangle=k(\! \! (J^{\star})\! \! )$ is an algebra of Noetherian series. By \Cref{prop-Noeth-salgebra}, it is an ultrafinite summability algebra for the structure of \Cref{prop-gloub-spaces}. Note that each $P \in k \langle\! \langle J \rangle\! \rangle$ is the sum $P = \sum_{\theta \in J^{\star}} P (\theta) X_{\theta}$.

    \subsection{Evaluating formal power series in summability algebras}\label{subsection-evaluations}

Given a set $J$, the summability algebra $k \langle\! \langle J \rangle\! \rangle$ extends the completion $k \left\langle J \right\rangle$ of the free associative algebra on $J$. It is well-known \cite[Chapter~0, p 17--18]{reutenauer} that elements of $k \langle J \rangle$ can be evaluated at tuples of elements in its maximal ideal. This allows for the development of a formal Lie correspondence for formal power series, based on the evaluation of the Taylor series of the exponential and logarithm.

It is very convenient to extend these results to more general summability algebras. Then one can see $k \langle\! \langle J \rangle\! \rangle$ as a universal and free summability algebra, acting by evaluation on summability algebras, so that universal identities (such as $\exp \circ \log = \tmop{id}$) that can be stated in summability algebras  could be proved once in $k \langle\! \langle J \rangle\! \rangle$ and then obtained in general by evaluating into summability algebras. 

\begin{definition}
  \label{def-closed}Let $(A,+,\cdot,\Sigma)$ be a unital ultrafinite
  summability algebra. Then we say that $(A,+,\cdot,\Sigma)$ 
  {\tmstrong{has evaluations}}{\index{evaluations}} if there exists a closed ideal $\mathfrak{m}$ of $A$ such 
  $A = k1_A + \mathfrak{m}$, and for all sets $J$ and all $\mathbf{f} \in
  \tmop{dom} \Sigma_J^{\mathfrak{m}}$, the family $( \mathbf{f} (\theta_1)
  \cdots \mathbf{f} (\theta_n))_{\theta = (\theta_1, \ldots, \theta_n) \in J^{\star}}$ is in $\tmop{dom} \Sigma_{J^{\star}}$.
\end{definition}

If $(A, \mathfrak{m}, \Sigma)$ are as in \Cref{def-closed}, then for all sets
$J$, all $\mathbf{f} \in \tmop{dom} \Sigma_J^{\mathfrak{m}}$ and all $P \in k \langle\! \langle J \rangle\! \rangle$, we define the
{\bf{evaluation}}{\index{evaluation}} of $P$ at $\mathbf{f}$ as\label{autolab23}
\[ \tmop{ev}_{\mathbf{f}} (P) \assign \sum_{\theta = (\theta_1, \ldots, \theta_n) \in
   J^{\star}} P (\theta) \mathbf{f} (\theta_1) \cdots \mathbf{f} (\theta_n) \in A. \]
If $J = \{0, \ldots, m - 1\}$ for an $m \in \mathbb{N}$, then we simply write
$\tmop{ev}_{\mathbf{f} (0), \ldots, \mathbf{f} (m - 1)} (P) \assign \tmop{ev}_{\mathbf{f}} (P)$.

\begin{remark}\label{rem-evaluations} Let us motivate \Cref{def-closed}.
Given a summability algebra $A$ and a family $\mathbf{f} \colon J \longrightarrow A$, the evaluation $\tmop{ev}_{\mathbf{f}}(P)$ of a formal power series $P \in k \langle\! \langle J \rangle\! \rangle$ at $\mathbf{f}$ ought to be the sum of the family $(P(j_1,...,j_n) \mathbf{f}(j_1) \cdots \mathbf{f}(j_n))_{(j_1,...,j_n) \in J^{\star}}$. Indeed, if all such families are summable, then $\tmop{ev}_{\mathbf{f}}$ will be the only strongly linear morphism of algebras between $k \langle\! \langle J \rangle\! \rangle$ and $A$ that sends $X_j$ to $\mathbf{f}(j)$ for each $j \in J$.

Taking $P = \sum \limits_{j \in J} X_j$, we see that this entails that $\mathbf{f}$ be summable. Furthermore, if $\mathbf{f}(j) \in k^{\times}$ for a certain $j \in J$, then for $P= \sum \limits_{k \in \mathbb{N}} (\mathbf{f}(j))^{-k} X^k_j$, we see that the  family $(P(j_1,...,j_n) \mathbf{f}(j_1) \cdots \mathbf{f}(j_n))_{(j_1,...,j_n) \in J^{\star}}$ has $(1)_{\theta \in \{j\}^{\star}}$ as a subfamily, hence is not summable. So if evaluations are to be defined in the case when $A=k \langle\! \langle J \rangle\! \rangle$, then $\mathbf{f}$ should range in the maximal ideal $k \langle\! \langle J \rangle\! \rangle_-$ of $k \langle\! \langle J \rangle\! \rangle$.

In general, writing $\mathfrak{m}$ for the set of elements $\tmop{ev}_{\mathbf{a}}(P) \in A$ where $J$ ranges among all sets, $P$ ranges among all elements in $k \langle\! \langle J \rangle\! \rangle_-$ and $\mathbf{f}$ ranges among summable families for which $(P(j_1,...,j_n) \mathbf{f}(j_1) \cdots \mathbf{f}(j_n))_{(j_1,...,j_n) \in J^{\star}}$ is summable, then $\mathfrak{m}$ is a two-sided ideal of the algebra $A'\assign k + \mathfrak{m}$.

Imposing furthermore that $\mathfrak{m}$ be a closed subspace of $A$, we might as well take $A=A'$ and thus work with summability algebras of the form $A=k+\mathfrak{m}$ for a two-sided ideal $\mathfrak{m}$ of $A$. We will see (\Cref{prop-closed-maximal-ideal}) that such a summaility algebra $k+\mathfrak{m}$ is local with maximal ideal $\mathfrak{m}$. \end{remark}

\begin{remark}\label{rem:subalgebraideal}
    Note that if $(A,\Sigma)$ is a unital summability algebra and $\mathfrak{m} \subseteq A$ is a subalgebra, then $k+\mathfrak{m}$ is a subalgebra of $A$ and $\mathfrak{m}$ is an ideal of $k+\mathfrak{m}$.
\end{remark}

Let $J$ be set. Let us check that $k \langle\! \langle J \rangle\! \rangle=k+k \langle\! \langle J \rangle\! \rangle_-$ itself has evaluations.

\begin{proposition}
  \label{prop-kJ-superclosed}The set $k \langle\! \langle J \rangle\! \rangle_-$
  is a closed ideal of $k \langle\! \langle J \rangle\! \rangle$, and $(k
  \langle\! \langle J \rangle\! \rangle, \Sigma)$ has evaluations.
\end{proposition}

\begin{proof}
  Let $I$ be a set. Let $\mathbf{Q}\colon I \longrightarrow k \langle\! \langle J \rangle \!
\rangle_0$ be summable. We have $\left( \sum_{i \in I}
  \mathbf{Q}(i) \right) (\emptyset) = \sum_{i \in \emptyset} \mathbf{Q}(i) (\emptyset) = 0$,
  so $\sum_{i \in I} \mathbf{Q}(i) \in k \langle\! \langle J \rangle\! \rangle_-$, which is
  thus a closed ideal.
  
  We next want to show that the family $( \mathbf{Q}(i_1) \cdots
  \mathbf{Q}(i_n))_{(i_1, \ldots, i_n) \in I^{\star}}$ is summable. Let $m \in \mathbb{N}$ and $\theta = (\theta_1, \ldots, \theta_m) \in J^m$.
  Write $I^{\star}_{\theta} \assign \{ i \in I^n \suchthat n \in \mathbb{N} \wedge
    \theta \in \tmop{supp}_m \mathbf{Q}(i_1) \cdots \mathbf{Q}(i_n) \}$ and $I_{\beta} \assign \{ i \in I \suchthat \beta \in \tmop{supp} \mathbf{Q}(i) \}$
  for all $\beta \in J^{\star}$. By \Cref{lem-support-multiproduct}, we have
  \[ I^{\star}_{\theta} \subseteq \bigcup_{n \in \mathbb{N}} \bigcup_{\beta_1 \cdots
    \beta_n = \theta} I_{\beta_1} \times \cdots \times I_{\beta_n}.\]
 Since each $\mathbf{Q}(i)$ lies in $k \langle\! \langle J
  \rangle\! \rangle_0$, we have $\tmop{supp}_0 \mathbf{Q}(i) = \emptyset$ for all $i \in
  I$, so we have in fact
  \[ I^{\star}_{\theta} \subseteq \bigcup_{n \leqslant m} \bigcup_{\beta_1
     \cdots \beta_n = \theta} I_{\beta_1} \times \cdots \times I_{\beta_n} .
  \]
  Now since each set $X_n \assign \{ (\beta_1, \ldots, \beta_n) \in
  (J^{\star})^n \suchthat \beta_1 \cdots \beta_n = \theta \}$ is finite
  and each $I_{\beta}, \beta \in J^{\star}$ is finite,
  we deduce that $I^{\star}_{\theta}$ is finite. Therefore the family $(\mathbf{Q}(i_1)
  \cdots \mathbf{Q}(i_n))_{i = (i_1, \ldots, i_n) \in I^{\star}}$ is
  summable.
\end{proof}

\begin{theorem} \label{th-evaluations}
  Assume that $A=k+\mathfrak{m}$ has evaluations. For all sets $J$ and all $\mathbf{f} \in \tmop{dom}
  \Sigma_J^{\mathfrak{m}}$, the evaluation map $\tmop{ev}_{\mathbf{f}}: k \langle\! \langle J \rangle\! \rangle
  \longrightarrow A$ is a morphism of algebras which is strongly linear.
\end{theorem}

\begin{proof}
  That $\tmop{ev}_{\mathbf{f}}$ is a linear map is a direct consequence of the fact that
  $\Sigma_{\{ (n, \theta) \suchthat n \in \mathbb{N} \wedge \theta \in J^n
  \}}^{\mathfrak{m}}$ is linear. For $P, Q \in k \langle \! \langle J \rangle \! \rangle$, \hyperref[sa]{\textbf{SA}} gives
  \begin{eqnarray*}
    \tmop{ev}_{\mathbf{f}} (P) \cdot \tmop{ev}_{\mathbf{f}} (Q) & = & \sum_{m, p \in \mathbb{N} \wedge \beta \in J^m \wedge \gamma \in
    J^p} P (\beta) Q (\gamma) \mathbf{f} (\beta_1) \cdots \mathbf{f} (\beta_m) \cdot \mathbf{f}
    (\gamma_1) \cdots \mathbf{f} (\gamma_p)\\
    & = & \sum_{n \in \mathbb{N} \wedge \theta \in J^n} (P \cdot Q) (\theta)
    \mathbf{f} (\theta_1) \cdots \mathbf{f} (\theta_n)\\
    & = & \tmop{ev}_{\mathbf{f}} (P \cdot Q) .
  \end{eqnarray*}
  So $\tmop{ev}_{\mathbf{f}}$ is a morphism of algebras.
  
  Let $\mathbf{P} \colon I \longrightarrow k \langle \! \langle J \rangle \! \rangle$ be summable and set $P \assign
  \sum_{i \in I}\mathbf{P}(i)$. Consider the element 
  \[ P' \assign \sum_{\theta \in J^{\star}} X_{\theta} \]
  of $k \langle \! \langle J \rangle \! \rangle$. Since $(A,+,\cdot,\Sigma)$ has evaluations, the family 
  $(P' (\theta) \mathbf{f} (\theta_1) \cdots \linebreak \mathbf{f} (\theta_n))_{\theta = (\theta_1,
  \ldots, \theta_n) \in J^{\star}}$ is summable. Consider the family
  of finite subsets of $k$
  \[ (C_{(n, \theta)})_{\theta = (\theta_1, \ldots, \theta_n) \in J^{\star}} =
     (\{ \mathbf{P}(i) (\theta) \suchthat \mathbf{P}(i)(\theta) \neq 0\} \})_{\theta = (\theta_1,
     \ldots, \theta_n) \in J^{\star}} . \]
  Recall that $(A, \Sigma)$ is ultrafinite, so $(cP' (\theta) \mathbf{f} (\theta_1)
  \cdots \mathbf{f} (\theta_n))_{\theta = (\theta_1, \ldots, \theta_n) \in J^{\star}
  \wedge c \in C_{(n, \theta)}}$ is summable in $(A, \Sigma)$. By \hyperref[ss2]{\textbf{SS2}},
  so is the family \[\mathbf{g} \assign (\mathbf{P}(i) (\theta) \mathbf{f} (\theta_1) \cdots \mathbf{f} (\theta_n))_{\theta =
  (\theta_1, \ldots, \theta_n) \in J^{\star} \wedge i \in I_{\theta}}.\] 
  By \hyperref[ss4]{\textbf{SS4}}, the family
  \[ (\tmop{ev}_{\mathbf{f}} (\mathbf{P}(i)))_{i \in I} = (\mathbf{P}(i) (\theta) \mathbf{f} (\theta_1) \cdots \mathbf{f}
     (\theta_n))_{\theta = (\theta_1, \ldots, \theta_n) \in J^{\star} \wedge i
     \in I} \]
  is summable as the union of the families $\mathbf{g}$ and $(0)_{\theta = (\theta_1,
  \ldots, \theta_n) \in J^{\star} \wedge i \in I \setminus I_{\theta}}$. By
  by ultrafiniteness of $(A, \Sigma)$, \Cref{prop-union-sum} gives
  \[ \sum_{i \in I} \tmop{ev}_{\mathbf{f}} (\mathbf{P}(i)) = \sum_{n \in \mathbb{N} \wedge \theta=(\theta_1,...,\theta_n)
     } \left( \sum_{i \in I_{\theta}} \mathbf{P}(i) (\theta) \right) \mathbf{f} (\theta_1)
     \cdots \mathbf{f} (\theta_n) = \tmop{ev}_{\mathbf{f}} (P) . \]
  This shows that $\tmop{ev}_{\mathbf{f}}$ is strongly linear.
\end{proof}

The previous theorem allows us to derive identities in a sumability algebra $A=k+\mathfrak{m}$ with evaluations from universal identities in $k \langle\! \langle J \rangle\! \rangle$ which only involve finite products and infinite sums. We apply this in the next two results.

The next proposition extends this to identities involving composition of formal power series (i.e.\ evaluations of formal power series at formal power series). We will apply this to the exponential and logarithmic series in the next subsection.

\begin{proposition}
  \label{cor-evaluation-associativity}Let $A=k+\mathfrak{m}$ be a summability algebra that has evaluations. Let $I,J$ be sets. Consider two families
  $\mathbf{Q} \colon I
  \longrightarrow k \langle\! \langle J \rangle\! \rangle_-$ and $\mathbf{f} \colon J \longrightarrow \mathfrak{m}$ which are summable in their respective algebras. For $P \in k \langle\! \langle I
  \rangle\! \rangle$, we have
  \[ \tmop{ev}_{\mathbf{f}} (\tmop{ev}_{\mathbf{Q}} (P)) = \tmop{ev}_{(\tmop{ev}_{\mathbf{f}} (\mathbf{Q}(i)))_{i \in I}} (P) . \]
\end{proposition}

\begin{proof}
  Let $i_0 \in I$. We have $\tmop{ev}_{\mathbf{Q}} (X_{i_0}) = \mathbf{Q}(i_0)$, so
  \[ \tmop{ev}_{\mathbf{f}} (\tmop{ev}_{\mathbf{Q}} (X_{i_0})) = \tmop{ev}_{\mathbf{f}} (\mathbf{Q}(i_0)) =
     \tmop{ev}_{(\tmop{ev}_{\mathbf{f}} (\mathbf{Q}(i)))_{i \in I}} (X_{i_0}) . \]
  Since $\tmop{ev}_{\mathbf{f}} \circ \tmop{ev}_{\mathbf{Q}}$ and $\tmop{ev}_{(\tmop{ev}_{\mathbf{f}} (\mathbf{Q}(i)))_{i \in I}}$ are strongly linear morphisms of algebras, we deduce
  that the identity holds for all~$P \in k
  \langle\! \langle I \rangle\! \rangle$.
\end{proof}

\begin{proposition}
  \label{prop-closed-maximal-ideal}Let $A = k +\mathfrak{m}$ have evaluations. Then $A$ is local algebra with maximal ideal $\mathfrak{m}$.
\end{proposition}

\begin{proof} We first show that each element of $1 +\mathfrak{m}$ is invertible in $A$.
  Let $a = 1 + \varepsilon \in 1 +\mathfrak{m}$. Since $(A,+,\cdot,\Sigma)$ has evaluations, writing
  \[ P \assign 1 - X_0 + X_0^2 - X_0^3 + \cdots \in k \langle \! \langle 1 \rangle \! \rangle, \]
  we may consider the evaluation $\tmop{ev}_{\varepsilon} (P) \in A$. We have $P(1+X_0)=(1+X_0) P=1$ in $k \langle \! \langle 1 \rangle \! \rangle$, so \Cref{th-evaluations} gives $\tmop{ev}_{\varepsilon}(P) \cdot (1+\varepsilon)= (1+\varepsilon)\cdot \tmop{ev}_{\varepsilon}(P)=1$, i.e.\ $1 + \varepsilon$ is a unit in $A$ with inverse $\tmop{ev}_{\varepsilon}
  (P)$.

We deduce that $k^{\times} (1 + \mathfrak{m}) = k^{\times}+\mathfrak{m}$ is contained in $U(A)$. Since $\mathfrak{m}$ is a proper ideal, it follows that $A \setminus U(A) = \mathfrak{m}$. So $A \setminus U(A)$ is an ideal, whence $A$ is local with maximal ideal $\mathfrak{m}$.
\end{proof}

    Due to \Cref{prop-closed-maximal-ideal}, if $(A,+,\cdot,\Sigma)$ has evaluations (see  \Cref{def-closed}), then there is already a unique $\mathfrak{m}$ with $A=k1_A+\mathfrak{m}$, which is the unique maximal ideal of $A$.

\section{Derivations and endomorphisms}\label{section-closed}

\subsection{Strongly linear derivations}

  Let $(A,+,\cdot,\Sigma)$ be a summability algebra.

\begin{definition}
  A {\tmstrong{strongly linear derivation}}{\index{strongly linear derivation}} on $A$
  is a strongly linear map $\partial \colon A \longrightarrow A$ which
  satisfies the Leibniz product rule
  \[ \forall a, b \in A, \partial (a \cdot b) = \partial (a) \cdot b + a \cdot
     \partial (b) . \]
  We write $\tmop{Der}^+ (A)$\label{autolab16} for the set of strong
  derivations on $A$.
\end{definition}

It follows from \Cref{prop-lin-subalgebra} that $\left( \tmop{Der}^+ (A), +,
\left[ \hspace{1.2em} \right], . \right)$ is a Lie algebra.

\begin{proposition}
  The subspace $\tmop{Der}^+ (A)$ of $\tmop{Lin}^+ (A)$ is closed under $\Sigma^{\tmop{Lin}}$.
\end{proposition}

\begin{proof}
  Let $I$ be a set, let $(\partial_i)_{i \in I}$ be a family in $\tmop{Der}^+ (A)$ which is
  $\tmop{Lin}$-summable, and write $\partial \assign \sum_{i \in I} \partial_i$. Let
  $a, b \in A$. We have
  \begin{align*}
    \partial (a \cdot b) & = \sum_{i \in I} \partial_i (a \cdot b)&&\\
    & = \sum_{i \in I} (\partial_i (a) \cdot b + a \cdot \partial_i (b))&&\\
    & = \sum_{i \in I} \partial_i (a) \cdot b + \sum_{i \in I} a \cdot
    \partial_i (b)&&\\
    & = \left( \sum_{i \in I} \partial_i (a) \right) \cdot b + a \cdot
    \left( \sum_{i \in I} \partial_i (b) \right) && \text{(\hyperref[sa]{\textbf{SA}})}\\
    & = \partial (a) \cdot b + a \cdot \partial (b). &&
  \end{align*}
  So $\partial \in \tmop{Der}^+ (A)$.
\end{proof}

\subsection{Exponential and logarithm}\label{section-exp-log}

In the sequel of \Cref{section-closed}, we assume that $k$ has characteristic
zero. It is known \cite[Chapter 4, Section 7]{Serre:Lie} that given a finite set $J$, the algebra $k \langle \! \langle J \rangle \! \rangle$  is equipped with an exponential $\exp \colon k \langle \! \langle J \rangle \! \rangle_0 \longrightarrow 1+k \langle \! \langle J \rangle \! \rangle_0$ and a logarithm $\log \colon 1+k \langle \! \langle J \rangle \! \rangle_0 \longrightarrow k \langle \! \langle J \rangle \! \rangle_0$, which are inverses of one another, and are given by evaluating the usual formal series. Using our previous results, we will recover a number of known identities known in the case of formal power series for all summability algebras with evaluations.

\begin{definition}\label{def:explog}
    Let $E_0,L_0\in k \langle \! \langle 1 \rangle \! \rangle$ be defined as
follows:
\[ E_0 \assign \sum_{n \geqslant 0} \frac{1}{n!} X_0^n
   \text{ and }L_0 \assign \sum_{n \geqslant 1}
   \frac{(- 1)^{n + 1}}{n} X_0^n  \]
    Let $(A,+,\cdot,\Sigma)$ be a unital ultrafinite summability algebra that has evaluations, and let $\mathfrak{m}$ be its maximal ideal. We define two maps
\begin{eqnarray*}
  \exp\colon \mathfrak{m} & \longrightarrow & 1 +\mathfrak{m}\\
  \delta & \longmapsto & \tmop{ev}_{\delta} (E_0)
\end{eqnarray*}
and
\begin{eqnarray*}
  \log\colon 1 +\mathfrak{m} & \longrightarrow & \mathfrak{m}\\
  \delta & \longmapsto & \tmop{ev}_{\delta - 1} (L_0) .
\end{eqnarray*}
\end{definition}
\begin{proposition}
  {\tmem{{\cite[Section~1.7, Theorem~7.2]{Serre:Lie}}}} We have
  \[\exp(L_0)-1 = \log(E_0) = X_0.\]
\end{proposition}

\begin{corollary}
  \label{cor-exp-log-bij}Let $(A,+,\cdot,\Sigma)$ be a summability algebra with evaluations, with
  maximal ideal $\mathfrak{m}$. Then $\exp \colon \mathfrak{m} \longrightarrow 1
  +\mathfrak{m}$ and $\log \colon 1 +\mathfrak{m} \longrightarrow \mathfrak{m}$ are
  bijective, and are functional inverses of one another.
\end{corollary}

\begin{proof}
  This follows from \Cref{cor-evaluation-associativity}. For instance, for $\varepsilon \in \mathfrak{m}$, we have \[\log(\exp(\varepsilon))=\tmop{ev}_{\tmop{ev}_{\varepsilon}(E_0-1)}(L_0)=\tmop{ev}_{\varepsilon}(\tmop{ev}_{L_0} (E_0-1)) =\tmop{ev}_{\varepsilon}(X_0)=\varepsilon.\] The other identity follows similarly.
\end{proof}

Next consider the following elements of $k \langle \!\langle 2 \rangle\! \rangle$ for $n \in
\mathbb{N}$:
\begin{eqnarray}
  K_n & \assign & \sum_{\substack{m_1 + p_1, \cdots, m_i + p_i \geqslant 1 \\ m_1+p_1+\ldots + m_i+p_i=n}} \frac{1}{m_1
  !p_1 ! \cdots m_n !p_n !} X_0^{m_1} X_1^{p_1} \cdots X_0^{m_n} X_1^{p_n}
  \  \label{eq-BCH-term}\\
  X_0 \ast X_1 & \assign & \sum_{n > 0} \frac{(- 1)^{n + 1}}{n} K_n .
  \nonumber
\end{eqnarray}
We have the formal Baker--Campbell--Hausdorff Theorem:

\begin{proposition}
  \label{prop-BCH}{\tmem{{\cite[Chapter~4, Section~8]{Serre:Lie}}}} We have $\exp(X_0 \ast X_1) = \exp(X_0) \cdot \exp(X_1)$. Moreover $K_0 = X_0 + X_1$ and each
  $K_n, n > 0$ lies in the Lie subalgebra of $(k \langle \! \langle 2 \rangle \! \rangle, +, 0,
  [\cdot, \cdot])$ generated by commutators of $X_0$ and $X_1$.
\end{proposition}

\begin{definition}\label{def:star}
    Let $(A,+,\cdot,\Sigma)$ be a unital ultrafinite summability algebra and let $\mathfrak{m}$ be a closed ideal of $A$. Suppose that $A=k 1_A+\mathfrak{m}$ and that $A$ has evaluations. For all $\delta_1$, $\delta_2$ in $\mathfrak{m}$, we define 
\begin{eqnarray}\delta_1 \ast \delta_2 & \assign & \tmop{ev}_{\delta_1,\delta_2}(X_0\ast X_1) \in \mathfrak{m}. \label{eq-BCH-prod}\end{eqnarray}
\end{definition}

As a consequence of \Cref{cor-evaluation-associativity}, we have:

\begin{corollary}
  \label{cor-BCH-operation}Let $(A,+,\cdot,\Sigma)$ be a summability algebra with evaluations, and let $\mathfrak{m}$ denote its maximal ideal. Then for all $\delta_1, \delta_2 \in
  \mathfrak{m}$, we have $\exp (\delta_1) \cdot \exp (\delta_2) = \exp
  (\delta_1 \ast \delta_2)$.
\end{corollary}

\subsection{A group isomorphism between derivations and automorphisms}

Let $(A,+,\cdot,\Sigma)$ be an ultrafinite, unital summability algebra. Let $\mathfrak{m} \subseteq \tmop{Lin}^+ (A)$ be a closed subalgebra such that $k \tmop{Id}_A + \ \mathfrak{m}$ has evaluations.

\begin{proposition}
  \label{prop-der-aut-correspondence}The exponential map $\exp\colon \tmop{Der}^+ (A) \cap \mathfrak{m} \longrightarrow \tmop{End}^+ (A) \cap
  \tmop{Id}_A +\mathfrak{m}, \ \partial \mapsto \exp(\partial)$ is a bijection with inverse $\log$.
\end{proposition}

\begin{proof} Our proof is a slight adaptation of {\cite[Theorem~4]{Praa86}} to our formal
  context. Let $\partial \in \tmop{Der}^+ (A) \cap \mathfrak{m}$ and let $a, b
  \in A$. For $n \in \mathbb{N}$, an easy induction using the Leibniz product
  rule shows that

  \[ \partial^{[n]} (a \cdot b) = \sum_{i = 0}^n \binom{n}{i} \partial^{[i]} (a)
     \cdot \partial^{[n - i]} (b) . \]
  We have
  \begin{align*}
    \exp (\partial) (a) \cdot \exp (\partial) (b) & = \left( \sum_{m \in
    \mathbb{N}} \frac{1}{m!} \partial^{[m]} (a) \right) \cdot \left( \sum_{p
    \in \mathbb{N}} \frac{1}{p!} \partial^{[p]} (b) \right) &&\\
    & = \sum_{m, p \in \mathbb{N}} \frac{1}{m!p!} \partial^{[m]} (a) \cdot
    \partial^{[p]} (b) && \text{(\hyperref[sa]{\textbf{SA}})}\\
    & = \sum_{n \in \mathbb{N}} \sum_{m + p = n} \frac{1}{m!p!}
    \partial^{[m]} (a) \cdot \partial^{[p]} (b) && \text{(\hyperref[ss3]{\textbf{SS3}})}\\
    & = \sum_{n \in \mathbb{N}} \sum_{i = 0}^n \frac{1}{i! (n - i) !}
    \partial^{[i]} (a) \cdot \partial^{[n - i]} (b) && \\
    & = \sum_{n \in \mathbb{N}} \frac{\partial^{[n]} (a \cdot b)}{k!} &&\\
    & = \exp (\partial) (a \cdot b). &&
  \end{align*}
  So $\exp (\partial) \in \tmop{End}^+ (A)$. Conversely, let $\sigma \in
  \tmop{End}^+ (A) \cap \tmop{Id}_A +\mathfrak{m}$ and write $\varepsilon
  \assign \tmop{Id}_A - \sigma \in \mathfrak{m}$. Let $a, b \in A$. As in the
  proof of {\cite[Theorem~4]{Praa86}}, there is a family $(c_{k, l, n})_{k, l,
  n \in \mathbb{N}} \in \mathbb{Q}^{\mathbb{N}^3}$ such that for all $n > 0$,
  we have both
  \begin{eqnarray*}
    \varepsilon^{[n]} (a \cdot b) & = & \sum_{l = 0}^n \sum_{m = 0}^l c_{n, l,
    m} \varepsilon^{[m]} (a) \cdot \varepsilon^{[l - m]} (b) \text{\qquad in
    $A$, and}\\
    (z_1 + z_2 - z_1 z_2)^n & = & \sum_{l = 0}^n \sum_{m = 0}^l c_{n, l, m}
    z_1^m z_2^{l - m} \text{\qquad in $\mathbb{Q} [[z_1, z_2]]$} .
  \end{eqnarray*}
  Note that given $l, m \in \mathbb{N}$ we have
  \begin{equation}
    \forall n > l + m, c_{n, l, m} = 0 \label{eq-aux-cnlm} .
  \end{equation}
  So the sum $S_{l, m} \assign \sum_{n = 1}^{+ \infty} \frac{1}{n} c_{n, l,
  m}$ has finite support. We have
  \[ \log (1 - (z_1 + z_2 - z_1 z_2)) = \log (1 - z_1) + \log (1 - z_2) \]
  in $\mathbb{Q} [[z_1, z_2]]$. Identifying in the left and right hand terms
  the coefficients of $z_1^p z_2^q$ for $p \neq 0$ and $q \neq 0$, we deduce
  ithat $S_{l, m} = 0$ if $m \nin \{0, l\}$ or $l = 0$. Considering the
  coefficients of $z_1^l$ and $z_2^l$ for $l \geqslant 1$, we see that $S_{l,
  0} = S_{l, 1} = \frac{1}{l}$ otherwise. Now \
  \begin{align*}
    \log (\sigma) (a \cdot b) & = \sum_{n \geqslant 1} \frac{1}{n}  \sum_{l
    = 0}^n \sum_{m = 0}^l c_{n, l, m} \varepsilon^{[m]} (a) \cdot
    \varepsilon^{[l - m]} (b)&& \\
    & = \sum_{n \geqslant 1} \sum_{m \leqslant l \leqslant n}^n \frac{c_{n,
    l, m}}{n} \varepsilon^{[m]} (a) \cdot \varepsilon^{[l - m]} (b)&& \\
    & = \sum_{n \geqslant 1 \wedge m \leqslant l \leqslant n} \frac{c_{n,
    l, m}}{n} \varepsilon^{[m]} (a) \cdot \varepsilon^{[l - m]} (b)
    && \text{(\hyperref[ss3]{\textbf{SS3}})}\\
    & = \sum_{l \geqslant 0 \wedge m \leqslant l \wedge n \geqslant l
    \wedge n \geqslant 1} \frac{c_{n, l, m}}{n} \varepsilon^{[m]} (a) \cdot
    \varepsilon^{[l - m]} (b) && \text{(\hyperref[ss2]{\textbf{SS2}})}\\
    & = \sum_{l \geqslant 0} \sum_{n \geqslant 1 \wedge n \geqslant l
    \geqslant m} \frac{c_{n, l, m}}{n} \varepsilon^{[m]} (a) \cdot
    \varepsilon^{[l - m]} (b) && \text{(\hyperref[ss3]{\textbf{SS3}})}\\
    & = \sum_{l \geqslant 0} \sum_{m \leqslant l} \sum_{n \geqslant 1
    \wedge n \geqslant l \wedge n \leqslant l + m} \frac{c_{n, l, m}}{n}
    \varepsilon^{[m]} (a) \cdot \varepsilon^{[l - m]} (b)
    && \text{((\ref{eq-aux-cnlm}))}\\
    & = \sum_{l \geqslant 0} \sum_{m \leqslant l} S_{l, m}
    \varepsilon^{[m]} (a) \cdot \varepsilon^{[l - m]} (b)\\
    & = \sum_{l \geqslant 1} \frac{1}{l}  \sum_{m \in \{0, l\}}
    \varepsilon^{[m]} (a) \cdot \varepsilon^{[l - m]} (b)
    && \\
    & = \sum_{l \geqslant 1} \frac{1}{l}  (a \cdot \varepsilon^{[l]} (b) +
    \varepsilon^{[l]} (a) \cdot b)&&\\
    & = a \cdot \log (1 - \varepsilon) (b) + \log (1 - \varepsilon) (a)
    \cdot b&&\\
    & = a \cdot \log (\sigma) (b) + \log (\sigma) (a) \cdot b.&&
  \end{align*}
  Therefore $\log (\sigma)$ is a derivation.
\end{proof}

\begin{theorem}
  \label{th-der-aut-correspondence}Let $(A,+,\cdot,\Sigma)$ be an ultrafinite summability algebra. Suppose that $\mathfrak{m} \subseteq \tmop{Lin}^+ (A)$ is a closed subalgebra such that $k \tmop{Id}_A + \ \mathfrak{m} \subseteq \tmop{Lin}^+ (A)$
  %
  has evaluations. Consider the operation $\ast$ of \Cref{def:star} on $\mathfrak{m}$. Then the structures $(\tmop{Der}^+ (A) \cap
  \mathfrak{m}, \ast)$ and $(\tmop{End}^+ (A) \cap (\tmop{Id}_A +\mathfrak{m}),
  \circ)$ are groups, and the exponential map of
  \Cref{def:explog} restricts to a group isomorphism $\exp\colon\tmop{Der}^+ (A) \cap
  \mathfrak{m} \longrightarrow \tmop{End}^+ (A) \cap (\tmop{Id}_A +\mathfrak{m})$.
\end{theorem}

\begin{proof}
  By \Cref{prop-closed-maximal-ideal}, each $\sigma \in \tmop{Id}_A
  +\mathfrak{m}$ is invertible in $\tmop{Id}_A +\mathfrak{m}$, and its inverse
  is obviously a morphism of algebra of $A$. Therefore $\tmop{End}^+ (A) \cap
  \tmop{Id}_A +\mathfrak{m}$ is a group under composition. We conclude with
  \Cref{cor-BCH-operation} and \Cref{prop-der-aut-correspondence}.\end{proof}

It is not clear at first glance that there are (non-trivial) examples of applications of \Cref{th-der-aut-correspondence}. Those will be provided with \Cref{th-der-aut-correspondence-Hahn-field}.

\begin{corollary}\label{cor-divisible-torsion-free}
  The group $(\tmop{End}^+ (A) \cap \tmop{Id}_A +\mathfrak{m}, \circ)$ is
  divisible and torsion-free.
\end{corollary}

\begin{proof} Let $n \in \mathbb{N}^{>0}$ and $a \in \tmop{Der}^+ (A) \cap
  \mathfrak{m}$. Note that the $n$-fold iterate of $a$ in $(\tmop{Der}^+ (A),\ast)$ is $a \ast a \cdot \ast a = n a$. It follows since $k$ has haracteristic zero that $(\tmop{Der}^+ (A) \cap \mathfrak{m},\ast)$ is torsion-free. Furthermore, we see that the $n$-fold iterate of $\frac{1}{n} a$ in $(\tmop{Der}^+ (A),\ast)$ is $a$, whence $(\tmop{Der}^+ (A) \cap \mathfrak{m},\ast)$ is divisible. We conclude with \Cref{th-der-aut-correspondence}.
\end{proof}

\begin{proposition}
  \label{prop-divisibility}For all $\sigma \in \tmop{Id}_A +\mathfrak{m}$,
  writing $\mathcal{C} (\sigma) = \{\mu \in \tmop{End}^+ (A) \cap \tmop{Id}_A
  +\mathfrak{m} \suchthat \mu \circ \sigma = \sigma \circ \mu\}$, we have a
  group morphism
  \begin{eqnarray*}
    \mathord{\:^{[\cdot]}} \colon (k, +, 0) & \longrightarrow &
    (\mathcal{C}(\sigma), \circ, \tmop{Id}_A)\\
    c & \longmapsto & \sigma^{[c]} \assign \exp (c \log (\sigma)),
  \end{eqnarray*}
  with $\sigma^{[1]} = \sigma$ and $(\sigma^{[c]})^{[c']} = \sigma^{[cc']}$
  for all $c, c' \in k$. It is injective if $\sigma \neq \tmop{Id}_A$.
\end{proposition}

\begin{proof}
  Write $\partial \assign \log (\sigma) \in \tmop{Der}_{\prec}^+ (\mathbb{A})$
  and let $c \in k$. Recall by \Cref{prop-BCH} that for $n \in \mathbb{N}^{>
  0}$, the terms $\tmop{ev}_{(\partial, c \partial)} (K_n)$ and $\tmop{ev}_{(c
  \partial, \partial)} (K_n)$, \ where $K_n$ is as in (\ref{eq-BCH-term}),
  lie in the Lie algebra generated by commutators in $\partial$ and $c
  \partial$. All such commutators are zero, so~$\tmop{ev}_{(\partial, c
  \partial)} (K_n) = \tmop{ev}_{(c \partial, \partial)} (K_n) = 0$. It follows
  since $\tmop{ev}_{(\partial, c \partial)} (K_0) = \partial + c \partial =
  \tmop{ev}_{(c \partial, \partial)} (K_0)$ that $(c \partial) \ast \partial = (c + 1)
  \partial = \partial \ast (c \partial)$, so $c \partial$ commutes with
  $\partial$, whence $\exp (c \partial) = \sigma^{[c]} \in \mathcal{C}
  (\sigma)$.
  
  For $c, c' \in k$, we have $\sigma^{[c + c']} = \exp (c \partial + c'
  \partial) = \exp ((c \partial) \ast (c' \partial))$ as above. So $\sigma^{[c
  + c']} = \exp (c \partial) \circ \exp (c' \partial) = \sigma^{[c]} \circ
  \sigma^{[c']}$. Thus $\mathord{\:^{[\cdot]}}$ is a group morphism. We also
  have $\sigma^{[cc']} = \exp (c' \log (\exp (c \log (\sigma)))) = \exp (cc'
  \log (\sigma)) = \sigma^{[cc']}$.
  
  Assume that $\sigma \neq \tmop{Id}_A$, so $\log (\sigma) \neq 0$. The kernel
  of the morphism is
  \[ \{c \in k \suchthat \exp (c \log (\sigma)) = \tmop{Id}_A \} = \{c \in k
     \suchthat c \log (\sigma) = 0\} = \{0\} . \]
  So this morphism is injective. 
\end{proof}

\section{Application to Noetherian series}\label{section-Noetherian-series}


    In this section, we show that \Cref{th-der-aut-correspondence} applies to certain summability algebras of strongly linear maps on spaces of Noetherian series, yielding our second main result \Cref{th-der-aut-correspondence-Hahn-field}. 
    We will give examples of ultrafinite summability algebras $(k \tmop{id}_{\mathbb{A}}+\mathfrak{m},+,\cdot,\Sigma)$ with evaluations, where $\mathbb{A}$ is a space of Noetherian series and $\mathfrak{m}$ is the space of contracting strongly linear maps on $\mathbb{A}$.
    From this we deduce our third and final main result \Cref{th-Lie}.
    \textbf{Throughout this section, let $(\Omega,<)$ be a non-empty partially ordered set.}
    
\subsection{Contracting linear maps}

    We now consider the space $(\mathbf{H}_{\Omega} k,\Sigma^{\Omega})$ of Noetherian series with exponents in $\Omega$ from \Cref{def:Noetheriansummability}. First, we introduce a partial ordering $\prec$ on $\mathbf{H}_{\Omega} k$ and the notion of contracting endomorphisms on $\mathbf{H}_{\Omega} k$.

    \begin{definition}
        Given $v, w \in \mathbf{H}_{\Omega} k$, we
        write\label{autolab26}
        \[ v \prec w \]
        if $w \neq 0$, and for any $p \in \tmop{supp} v$ there is $q \in \tmop{supp}
        w$ with $p > q$.
    \end{definition}

    \begin{proposition}
        \label{prop-prec-elem} The
        relation $\prec$ is a partial ordering on $\mathbf{H}_{\Omega} k$. Moreover, for any $u, v, w \in
        \mathbf{H}_{\Omega} k$, if $u \prec w$ and $v \prec w$, then $u + v \prec
        w$.   
    \end{proposition}

    \begin{proof} Let $u, v, w \in
        \mathbf{H}_{\Omega} k $. We have $0 \not \prec 0$ by definition. If $u \neq 0$, then there is a minimal element $p \in \tmop{supp} u$. Thus, $u \not \prec u$.
        Now suppose that $u \prec v$ and $v \prec w$. Then for any $p \in \tmop{supp} u$ there are $q\in \tmop{supp} v$ and $r \in \tmop{supp} w$ with $p>q>r$. Thus, $p>r$ and $u \prec w$. Thus, $\prec$ is a partial ordering on $\mathbf{H}_{\Omega} k$. 
        Next suppose that $u\prec w$ and $v \prec w$. Since $\tmop{supp}(u+v) \subseteq \tmop{supp}(u)\cup \tmop{supp}(v)$, we immediately obtain $u+v \prec w$. 
    \end{proof}

    \begin{definition}
        \label{def-contracting}An endomorphism $\phi\in\mathrm{Lin}(\mathbf{H}_{\Omega} k)$ is {\bf{{\tmstrong{contracting}}}}{\index{contracting
        map}} if for any $v \in \mathbf{H}_{\Omega} k \setminus \{0\}$, we have $\phi(v) \prec v$.
        We write $\tmop{Lin}_{\prec} (\mathbf{H}_{\Omega} k)$\label{autolab27} (respectively
        $\tmop{Lin}_{\prec}^+ (\mathbf{H}_{\Omega} k)$\label{autolab28}) for the set of
        contracting linear (contracting strongly linear) endomorphisms on $\mathbf{H}_{\Omega} k$. 
    \end{definition}

    Note that
    \begin{equation}
        \tmop{Lin}_{\prec}^+ (\mathbf{H}_{\Omega} k) = \tmop{Lin}_{\prec} (\mathbf{H}_{\Omega} k) \cap
        \tmop{Lin}^+ (\mathbf{H}_{\Omega} k) . \label{eq-Lin-cont}
    \end{equation}

    \begin{lemma}
        \label{lem-Lin-cont-subalgebra}$\tmop{Lin}_{\prec} (\mathbf{H}_{\Omega} k)$ is a
        subalgebra of $(\tmop{Lin} (\mathbf{H}_{\Omega} k),+,\circ)$. 
    \end{lemma}

    \begin{proof}
        Let $\phi, \psi \in \tmop{Lin}_{\prec} (\mathbf{H}_{\Omega} k)$, let $c \in k$ and let 
        $v \in \mathbf{H}_{\Omega} k \setminus \{ 0 \}$. Then $\tmop{supp} c
        \psi (v) \subseteq \tmop{supp} \psi (v)$, implying $c
        \psi (v) \prec v$. By \Cref{prop-prec-elem}, we obtain $\phi (v) + c
        \psi (v) \prec v$, whence $\phi + c \psi \in \tmop{Lin}_{\prec}
        (\mathbf{H}_{\Omega} k)$. If $\psi (v) = 0$, then $\phi (\psi (v)) = 0 \prec v$.
        Otherwise $\phi (\psi (v)) \prec \psi (v) \prec v$. By \Cref{prop-prec-elem}, we obtain $(\phi
        \circ \psi) (v) \prec v$. Thus, in both cases $\phi \circ \psi
        \in \tmop{Lin}_{\prec} (\mathbf{H}_{\Omega} k)$.
    \end{proof}

    \begin{lemma}
        \label{lem-contracting-criterion}
        An endomorphism $\phi\in \tmop{Lin}^+ (\mathbf{H}_{\Omega} k)$ is contracting if and only if
        $\tmop{supp} \phi (\mathbbm{1}_p) > p$ for any $p \in \Omega$.
    \end{lemma}

    \begin{proof}
        Suppose that $\varphi$ is contracting and let $p\in \Omega$. Since $\varphi(\mathbbm{1}_p)\prec \mathbbm{1}_p$, we obtain for any $r\in \mathrm{supp}\,\varphi(\mathbbm{1}_p)$ that $r>p$, as required. Conversely, suppose that for any $p\in\Omega$ we have $\tmop{supp} \phi (\mathbbm{1}_p) > p$. Let $v = \sum_{q\in \Omega}v_q\mathbbm{1}_q \in \mathbf{H}_{\Omega} k\setminus\{0\}$.
        Let $r\in \tmop{supp},\varphi(v)$. Then by (\ref{eq:suppunion}) there is a $q \in \tmop{supp}\,v$ with $r \in \tmop{supp}\, \varphi(v_q \mathbbm{1}_q) = \tmop{supp}\, \varphi(\mathbbm{1}_q)$. So $r>q$ for some $q\in \mathrm{supp}\,v$. Hence, $\varphi(v)\prec v$, as required.
    \end{proof}


    \begin{corollary}\label{cor:linideal}
    $\tmop{Lin}_{\prec}^+ (\mathbf{H}_{\Omega} k)$ is a closed subalgebra of $(\tmop{Lin}^+
    (\mathbf{H}_{\Omega} k),+,\circ,\Sigma^{\mathrm{Lin}})$ and an ideal of $(k \tmop{Id}_{\mathbf{H}_{\Omega} k} +
        \tmop{Lin}_{\prec}^+ (\mathbf{H}_{\Omega} k),+,\circ,\Sigma^{\mathrm{Lin}})$.
    \end{corollary}

    \begin{proof}
        Let us show that  $\tmop{Lin}_{\prec}^+ (\mathbf{H}_{\Omega} k)$ is closed. Let $\boldsymbol{\phi} \colon J \longrightarrow \tmop{Lin}_{\prec}^+ (\mathbf{H}_{\Omega} k)$ be $\tmop{Lin}$-summable and set \[\sigma \assign \sum_{j
        \in J} \boldsymbol{\phi}(j) \in \tmop{Lin}^+ (\mathbf{H}_{\Omega} k).\] Let $p \in \Omega$. We have
        $\tmop{supp} \sigma (\mathbbm{1}_p) \subseteq \bigcup_{j \in J} \tmop{supp}
        \boldsymbol{\phi}(j) (\mathbbm{1}_p)$ 
        by (\ref{eq:suppunion}). \Cref{lem-contracting-criterion} gives $\tmop{supp}
    \boldsymbol{\phi}(j) (\mathbbm{1}_p) > p$ for each $j \in J$, so $\tmop{supp}             \boldsymbol{\sigma}
        (\mathbbm{1}_p) > p$, so $\sigma$ is contracting, again by
        \Cref{lem-contracting-criterion}. That $\tmop{Lin}_{\prec}^+ (\mathbf{H}_{\Omega} k)$ is a subalgebra follows from  (\ref{eq-Lin-cont}), 
        \Cref{prop-lin-subalgebra} and
        \Cref{lem-Lin-cont-subalgebra}. We conclude with \Cref{rem:subalgebraideal}.
    \end{proof}

\begin{lemma}
  \label{lem-Noeth-criterion} Let $I$
  be a set and let $\mathbf{f} \colon I \longrightarrow \mathbf{H}_{\Omega} k$
  be a map. Consider the set
  \[ N_{\mathbf{f}} \assign \{ (i, p) \in I \times \Omega \suchthat p \in \tmop{supp} \mathbf{f}
     (i) \}, \]
 partially ordered by $(i, p) <_{\mathbf{f}} (j, q) \Longleftrightarrow p < q$. Then $\mathbf{f}$ is
  summable if and only if $(N_{\mathbf{f}}, <_{\mathbf{f}})$ is Noetherian.
\end{lemma}

    \begin{proof}
        Consider a non-empty chain $C$ in $(N_{\mathbf{f}}, <_{\mathbf{f}})$. Given $(i, p) \in C$, we
  have $p \in \bigcup_{j \in I} \tmop{supp} \mathbf{f} (j)$, and $(i, p)$ is $<_{\mathbf{f}}$
  minimal in $C$ if and only if $p$ is minimal in $\bigcup_{i \in I}
  \tmop{supp} \mathbf{f} (i)$. So $N_{\mathbf{f}}$ has infinite decreasing chains if and only if
  $\bigcup_{i \in I} \tmop{supp} \mathbf{f} (i)$ has infinite decreasing chains.
  
  Consider an antichain $A$ in $(N_{\mathbf{f}}, <_{\mathbf{f}})$. For $(i, p), (j, q) \in A$,
  either $p = q$ and $i \neq j$ or $p$ and $q$ are not comparable in $(\Omega,
  <)$. So $N_{\mathbf{f}}$ has an infinite antichain if and only if there is an infinite
  antichain in $\bigcup_{i \in I} \tmop{supp} \mathbf{f} (i)$ or there is an $p \in
  \bigcup_{i \in I} \tmop{supp} \mathbf{f} (i)$ such that the set $I_p = \{ i \in I
  \suchthat p \in \tmop{supp} \mathbf{f} (i) \}$ is infinite.
  
  In view of the definitions of summability and Noetherian
  orderings, we deduce that $\mathbf{f}$ is summable if and only if
  $(N_{\mathbf{f}}, <_{\mathbf{f}})$ is Noetherian.
\end{proof}

\subsection{Extensive Noetherian choice operators}

    We next present in \Cref{th-vdH} a version of {\cite[Theorem~A.4]{vdH:phd}} adjusted to our setting and with a self-contained proof. This result is used to prove  \Cref{th-superclosed-End}, which stipulates that the hypothesis of \Cref{th-der-aut-correspondence} holds for algebras of Noetherian series. By combining Theorems~\ref{th-der-aut-correspondence} and \ref{th-superclosed-End} we obtain our second main result, \Cref{th-der-aut-correspondence-Hahn-field}, on the correspondence between derivations and automorphisms.


    \begin{definition} 
        Let $(\Omega,<)$ be a partially ordered set. 
        We call a map $\vartheta\colon \Omega \longrightarrow \mathcal{P}(\Omega)$ a \textbf{choice operator} (on $\Omega$). 
        A choice operator $\vartheta$ is  {\bf{strictly extensive}} if for any $p \in \Omega$, we have $p<\vartheta(p)$, i.e.\ $p<y$ for any $y\in \vartheta(p)$.
        We say that $\vartheta$ is {\bf{Noetherian}}{\index{Noetherian
        choice operator}} if for any Noetherian subset $Y \subseteq \Omega$, the subset
        \[ Y_{\vartheta} \assign \bigcup_{y \in Y} \vartheta (y) \subseteq \Omega \]
        is Noetherian. 
    \end{definition}

For any non-empty word $w = (w_0, \ldots, w_m) \in \Omega^{\star} \setminus \{
\emptyset \}$, we write $w_{\bullet} \assign w_m \in \Omega$ for the last
letter of $w$. Let $Y \subseteq \Omega$ be a subset and $\vartheta$ a choice operator on $\Omega$. Let $\vartheta^+ (Y)$
denote the set of non-empty words $(w_0, \ldots, w_m) \in \Omega^{\star}
\setminus \{ \emptyset \}$ where $w_0 \in Y$, and for each $i < m$, we have
$w_{i + 1} \in \vartheta (w_i)$. We endow $\vartheta^+ (Y)$ with the partial ordering
$<_{\vartheta}$ defined by
\[ w <_{\vartheta} w' \Longleftrightarrow w_{\bullet} < w'_{\bullet} . \]


\begin{proposition}
  \label{th-vdH} Let $(\Omega,<)$ be a partially ordered set.  Let $\vartheta$ be a
  Noetherian and strictly extensive choice operator on $\Omega$. Then for all
  Noetherian subsets $Y$ of $\Omega$, the partially ordered set $(\vartheta^+ (Y),<_{\vartheta})$ is Noetherian.
\end{proposition}

    \begin{proof}
Let $Y\subseteq \Omega$ be a Noetherian subset. Assume for contradiction that $(\vartheta^+ (Y),<_{\vartheta})$ is not Noetherian. By Lemmas~\ref{lem-Noeth-subsequence} and \ref{lem:minbadsequence}, there is a
bad sequence $(w_i)_{i \in \mathbb{N}}$ in $\vartheta^+
(Y)$ which is minimal for the length map $w \mapsto | w |$. We write $w_i = (w_{i,0},...,w_{i,|w_i|-1})$ for each $i \in \mathbb{N}$. Since $Y$ itself is Noetherian, there is an $m \in \mathbb{N}$ with $| w_j | \geqslant 2$ for all $j \geqslant m$. Otherwise their would be a subsequence of $(w_i)_{i \in \mathbb{N}}$ whose elements have length $1$, whence a bad sequence in $(Y,<_{\vartheta})=(Y,<)$.

For $j \geqslant m$, we write $z_j \assign (w_{j, 0}, \ldots, w_{j, | w_j | -
2}) \in \vartheta^+ (Y)$.  We claim that the set $\mathcal{Z} \assign \{ z_j
\suchthat j \geqslant m \}$, is Noetherian. Indeed, assume
for contradiction that it is not. By \Cref{lem-Noeth-subsequence}, there is a bad sequence $(z_{j_i})_{i \in \mathbb{N}}$ in
$\mathcal{Z}$, where $j_0 < j_1 < \cdots$. Note that by minimality of $(w_i)_{i \in \mathbb{N}}$, the sequence
\[ u \assign (w_0, \ldots, w_{j_0 - 1}, z_{j_0}, z_{j_1}, \ldots) \] cannot be not bad. Since $(z_{j_i})_{i \in
\mathbb{N}}$ and $(w_i)_{i \in
\mathbb{N}}$ are bad, there must exist $i < j_0$ and $p \in \mathbb{N}$ with
$w_i \leqslant_{\vartheta} z_{j_p}$. Since $\vartheta$ is strictly extensive,
we have
\[ (w_{j_p})_{\bullet} \in \vartheta ((z_{j_p})_{\bullet}) >
   (z_{j_p})_{\bullet}, \]
so $w_i <_{\vartheta} w_{j_p}$: a contradiction. We deduce that $\mathcal{Z}$ is
Noetherian. We may thus extract a nondecreasing subsequence $(z_{m_i})_{i \in \mathbb{N}}$ of $(z_j)_{j \in \mathbb{N}}$. The subset $\mathcal{Z}_{\bullet}$ of $\Omega$ is Noetherian, so $\vartheta(\mathcal{Z}_{\bullet})$ is Noetherian. Thus there is a subsequence $(w_{n_i})_{i \in \mathbb{N}}$ of $(w_{m_i})_{i \in \mathbb{N}}$ such that $((w_{n_i})_{\bullet})_{i \in \mathbb{N}}$ is nondecreasing.
Since $(z_{n_i})_{i \in \mathbb{N}}$ and $((w_{n_i})_{\bullet})_{i \in \mathbb{N}}$ are nondecreasing, so is $(w_{n_i})_{i \in \mathbb{N}}$.
This last contradiction concludes the proof.\end{proof}

\begin{lemma}
    $k \tmop{Id}_{\mathbf{H}_{\Omega} k} +
        \tmop{Lin}_{\prec}^+ (\mathbf{H}_{\Omega} k)$ is a closed subspace of $(\tmop{Lin}^+ (\mathbf{H}_{\Omega} k),\Sigma)$. 
\end{lemma}

\begin{proof}
    Let $\mathbf{f}$ be a $\tmop{Lin}$-summable family in $k \tmop{Id}_{\mathbf{H}_{\Omega} k} +
        \tmop{Lin}_{\prec}^+ (\mathbf{H}_{\Omega} k)$, indexed by a set $J$ and let $p \in \Omega$. The set \[J'\assign \{j \in J \suchthat \mathbf{f}(j) \notin \tmop{Lin}_{\prec}^+ (\mathbf{H}_{\Omega} k) \}= \{j \in J \suchthat p \in \tmop{supp} \mathbf{f}(j)(\mathbbm{1}_p)\} \] is finite by summability of $(\mathbf{f}(j)(\mathbbm{1}_p))_{j \in J}$. Note that $\sum \limits_{j \in J\setminus J'} \mathbf{f}(j) \in \tmop{Lin}_{\prec}^+ (\mathbf{H}_{\Omega} k)$ by \Cref{cor:linideal}. So by \hyperref[ss3c]{\textbf{SS3c}}, we have \[\sum \limits_{j \in J} \mathbf{f}(j) = \sum \limits_{j \in J'} \mathbf{f}(j)+\sum \limits_{j \in J\setminus J'} \mathbf{f}(j) \in k \tmop{Id}_{\mathbf{H}_{\Omega} k} +
        \tmop{Lin}_{\prec}^+ (\mathbf{H}_{\Omega} k).\]  This shows that $k \tmop{Id}_{\mathbf{H}_{\Omega} k} +
        \tmop{Lin}_{\prec}^+ (\mathbf{H}_{\Omega} k)$ is closed.
\end{proof}

It follows that $k \tmop{Id}_{\mathbf{H}_{\Omega} k} +
        \tmop{Lin}_{\prec}^+ (\mathbf{H}_{\Omega} k)$ is a summability algebra for the restricted summability structure.

\begin{theorem}
  \label{th-superclosed-End} The algebra $k \tmop{Id}_{\mathbf{H}_{\Omega} k} +
  \tmop{Lin}_{\prec}^+ (\mathbf{H}_{\Omega} k)$ has evaluations.
\end{theorem}

\begin{proof} In view of \Cref{cor:linideal}, it suffices to show   that given a set $J$ and a $\tmop{Lin}$-summable family $(\phi_j)_{j \in J}
  \subseteq (\tmop{Lin}_{\prec}^+ (\mathbf{H}_{\Omega} k))^J$, the
  family
  \[ (\phi_{\theta_1} \circ \cdots \circ \phi_{\theta_n})_{(\theta_1, \ldots,
     \theta_n) \in J^{\star}} \]
  is $\tmop{Lin}$-summable.
  We may assume that $J$ is non-empty, and it suffices to show that
  \[ (\phi_{\theta_0} \circ \cdots \circ \phi_{\theta_n})_{(\theta_0, \ldots,
     \theta_n) \in J^{\star} \setminus \{\varnothing\}} \]
  is $\tmop{Lin}$-summable.
  
  Let $\pi \colon \Omega \times J \longrightarrow \Omega$ denote the projection on
  the first variable. We consider the partial ordering $<$ on $\Omega \times J$ given
  by $(p, i) < (q, j) \Longleftrightarrow p < q$. Consider the choice operator
  $\vartheta$ on $\Omega \times J$ given by
  \[ \forall p \in \Omega, \vartheta (p, i) \assign \{ (q, j) \suchthat j \in
     J \wedge q \in \tmop{supp} \phi_j (\mathbbm{1}_p) \} . \]
  We want to apply Proposition~\ref{th-vdH} to this choice operator. Since each $\phi_j$ for $j \in J$ is contracting, this is a strictly
  extensive choice operator. Let $Y \subseteq \Omega \times J$ be Noetherian.
  Let us show that the set
  \[ Y_{\vartheta} = \{ (q, j) \suchthat j \in J \wedge (\exists p \in \pi
     (Y), (q \in \tmop{supp} \phi_j (\mathbbm{1}_p))) \} \]
  is Noetherian by showing that each sequence in $Y_{\vartheta}$ has an
  increasing subsequence. Let $(q_n, j_n)_{n \in \mathbb{N}}$ be a sequence in
  $Y_{\vartheta}$, and pick for each $n \in \mathbb{N}$ a $(p_n, i_n) \in Y$
  with $q_n \in \tmop{supp} \phi_{j_n} (\mathbbm{1}_{p_n})$. Since $Y$ is
  Noetherian, we may assume by \Cref{prop-Noetherian-subsequence} that $(p_n,
  i_n)_{n \in \mathbb{N}}$ is increasing.
  
  Assume that $(p_n)_{n \in \mathbb{N}}$ has a constant subsequence. Without
  loss of generality, we may assume that it is constant itself. Assume that
  $(j_n)_{n \in \mathbb{N}}$ has no injective subsequence, i.e.\ $\{j_n\mid n\in\mathbb{N}\}$ is finite. So it has a
  constant subsequence, and we may assume that it is constant. The sequence
  $(q_n)_{n \in \mathbb{N}}$ in the Noetherian set $\tmop{supp}
  \phi_{j_n} (\mathbbm{1}_{ p_n})=\tmop{supp}
  \phi_{j_0} (\mathbbm{1}_{ p_0})$ has an increasing
  subsequence $(q_{\psi (n)})_{n \in \mathbb{N}}$ by \Cref{prop-Noetherian-subsequence}. So $(q_{\psi(n)}, j_0)_{n \in \mathbb{N}}$ is an increasing
  subsequence of~$(q_n, j_n)_{n \in \mathbb{N}}$. Assume now that $(j_n)_{n
  \in \mathbb{N}}$ has an injective subsequence. Since $(\phi_{j_n})_{n \in
  \mathbb{N}}$ is $\tmop{Lin}$-summable, the family $(\phi_{j_n} (\mathbbm{1}_{ p_0
  }))_{n \in \mathbb{N}}$ is summable. Therefore $(q_n)_{n \in
  \mathbb{N}}$ has an increasing subsequence. If this subsequence has no strictly increasing subsequence, then it has a constant subsequence $(q_{\mu(n)})_{n \in \mathbb{N}}$, but then the set $\{ n \in \mathbb{N} \suchthat q_{\mu(0)} \in \tmop{supp} \phi_{j_{\mu(n)}} (\mathbbm{1}_{p_0
  })\}$ is infinite: contradicting the summability of $(\phi_{j_{\mu(n)}} (\mathbbm{1}_{p_0})
  )_{n \in \mathbb{N}}$.
  Therefore $(q_n,
  j_n)_{n \in \mathbb{N}}$ has a strictly increasing subsequence.
  
  Assume now that $(p_n)_{n \in \mathbb{N}}$ has no constant subsequence. So
  $(p_n)_{n \in \mathbb{N}}$ has a strictly increasing subsequence
  $(p_{\varphi (n)})_{n \in \mathbb{N}}$. The family $(\phi_{j_{\varphi (n)}}
  (\mathbbm{1}_{p_{\varphi (n)}}))_{n \in \mathbb{N}}$ is
  summable. In particular $(q_{\varphi (n)})_{n \in \mathbb{N}}$
  has a strictly increasing subsequence, whence again $(q_n, j_n)_{n \in
  \mathbb{N}}$ has a strictly increasing subsequence.
  
  This shows that $\vartheta$ is Noetherian. Let $S \subseteq \Omega$ be a
  Noetherian subset and pick an arbitrary $j \in J$. By \Cref{th-vdH}, the set $\vartheta^+ (S \times \{ j \})$ is Noetherian for
  $<_{\vartheta}$. This means by \Cref{lem-Noeth-criterion} that the family $(p_{\bullet})_{p \in \vartheta^+ (S \times \{ j \})}$ is
  summable. Write
  \[ I \assign \{ (p, n, \theta_0, \ldots, \theta_n) \suchthat p \in S \wedge
     n \in \mathbb{N} \wedge \theta_0, \ldots, \theta_n \in J \} . \]
  For $i = (p, n, \theta_0, \ldots, \theta_n) \in I$, let $W_i$ denote the set
  of words \[w = (w_0, \ldots, w_{n + 1}) \in \vartheta^+ (S \times \{ j \})\]
  where $w_0 = (p, j)$ and $w_{k + 1} = (p_{k + 1}, \phi_{\theta_k})$ for a
  $p_{k + 1} \in \tmop{supp} \phi_{\theta_k} (\mathbbm{1}_{\pi (w_k)})$. Note
  that for each $i = (p, n, \phi_0, \ldots, \phi_n) \in I$, there is a map
  $c \in k^{W_i}$ such that $\phi_{\theta_n} \circ \cdots \circ \phi_{\theta_0} (\mathbbm{1}_p) = \sum_{w
  \in W_i} c (w) \mathbbm{1}_{\pi (w_{\bullet})}$.
  
  The sets $W_i, i \in I$ are pairwise disjoint, so by \hyperref[ss4]{\textbf{SS4}},
  \hyperref[ss3b]{\textbf{SS3b}} and \hyperref[uss]{\textbf{UF}}, the family $\left(
  \sum_{w \in W_i} c (w) \mathbbm{1}_{ \pi (w_{\bullet}) } \right)_{i \in I} =
  (\phi_{\theta_0} \circ \cdots \circ \phi_{\theta_n} (\mathbbm{1}_p))_{p \in S \wedge n
  \in \mathbb{N} \wedge \theta_0, \ldots, \theta_n \in J}$ is
  summable. We conclude with
  \Cref{prop-multisummability-criterion}.
\end{proof}

\subsection{The Der-Aut correspondence for Noetherian series}

Here, we assume that $k$ has characteristic zero. Let $(\Omega,<)$ be a partially ordered set.
By \Cref{th-superclosed-End}, the summability algebra $k\tmop{Id}_{\mathbf{H}_{\Omega} k}+\tmop{Lin}_{\prec}^+
(\mathbf{H}_{\Omega} k)$ has evaluations. As a consequence of Corollaries~\ref{cor-exp-log-bij}
and~\ref{cor-BCH-operation}, we have a group operation
\begin{eqnarray*}
  \ast \colon \tmop{Lin}_{\prec}^+ (\mathbf{H}_{\Omega} k) \times \tmop{Lin}_{\prec}^+
  (\mathbf{H}_{\Omega} k) & \longrightarrow & \tmop{Lin}_{\prec}^+ (\mathbf{H}_{\Omega} k)\\
  (\phi, \psi) & \longmapsto &  \phi
  + \psi + \frac{1}{2}  (\phi \circ \psi - \psi \circ \phi) + \cdots,
\end{eqnarray*}
and a group isomorphism
\begin{eqnarray}
  \exp \colon (\tmop{Lin}_{\prec}^+ (\mathbf{H}_{\Omega} k), \ast) & \longrightarrow &
  (\tmop{Id}_{\mathbf{H}_{\Omega} k} + \tmop{Lin}_{\prec}^+ (\mathbf{H}_{\Omega} k), \circ) 
  \label{eq-iso}\\
  \phi & \longmapsto & \sum_{n \in \mathbb{N}} \frac{1}{n!} \phi^{[n]} .
  \nonumber
\end{eqnarray}
Let $\mathbb{A}= k(\!(M)\!)$ be an algebra of Noetherian series, where $M$ is a partially ordered monoid. Let $1$-$\tmop{Aut}^+_k (\mathbb{A})$ denote the group
of bijective strongly linear morphisms of algebra $\sigma \colon \mathbb{A} \longrightarrow \mathbb{A}$ which preserve products, and with $\sigma (a) - a \prec a$ for all $a \in \mathbb{A}
\setminus \{0\}$.

\begin{remark}
\label{rem:1aut}In the case when $M$ is a linearly ordered abelian group, our notation $1$-$\tmop{Aut}^+_k (\mathbb{A})$ is compatible with that of \cite{KuhlSer:vaut}.
\end{remark}

\begin{theorem}
  \label{th-der-aut-correspondence-Hahn-field}The structures
  $(\tmop{Der}^+_{\prec} (\mathbb{A}), \ast)$ and $\left( 1 \mathord{\text{-}}
  \tmop{Aut}^+_k (\mathbb{A}), \circ \right)$ are groups, and we have an
  isomorphism
  \begin{eqnarray*}
    \exp \colon (\tmop{Der}^+_{\prec} (\mathbb{A}), \ast) & \longrightarrow &
    \left( 1 \mathord{\text{-}} \tmop{Aut}^+_k (\mathbb{A}), \circ \right)\\
    \partial & \longmapsto & \sum_{n \in \mathbb{N}} \frac{\partial^{[n]}}{n!}
    .
  \end{eqnarray*}
\end{theorem}

    \begin{proof}
       By \Cref{th-superclosed-End}, we may apply Theorem~\ref{th-der-aut-correspondence}, for $\mathfrak{m}=\tmop{Lin}_{\prec}^+ (\mathbb{A})$, to the summability algebra
        $k \tmop{Id}_{\mathbb{A}} +
  \tmop{Lin}_{\prec}^+ (\mathbb{A})$. We thus obtain the group isomorphism $\exp$ between $\tmop{Lin}_{\prec}^+ (\mathbb{A})\cap\tmop{Der}^+ (\mathbb{A}) = \tmop{Der}^+_{\prec} (\mathbb{A})$ with the operation $\ast$, and $(\tmop{Id}_{\mathbb{A}} + \tmop{Lin}_{\prec}^+ (\mathbb{A})) \cap \tmop{End}^+(\mathbb{A}) = 1 \mathord{\text{-}} \tmop{Aut}^+_k (\mathbb{A})$ with the composition.
    \end{proof}

We finish with a formal analog of the Lie homomorphism theorem.

\begin{theorem}
  Let $W$ be an ordered set, and consider the space of Noetherian series $\mathbb{W} \assign \mathbf{H}_{\Omega} k$. Let $\Phi \colon \tmop{Lin}^+_{\prec}
  (\mathbb{V}) \longrightarrow \tmop{Lin}^+_{\prec} (\mathbb{W})$ be a
  strongly linear morphism of Lie algebras. Then there exists a unique group
  morphism $\Psi \colon \tmop{Id}_{\mathbb{V}} + \tmop{Lin}_{\prec}^+ (\mathbb{V})
  \longrightarrow \tmop{Id}_{\mathbb{W}} + \tmop{Lin}_{\prec}^+ (\mathbb{W})$
  with
  \[ \Psi (\exp (\phi)) = \exp (\Phi (\phi)) \]
  for all $\phi \in \tmop{Lin}^+_{\prec} (\mathbb{V})$.
\end{theorem}

\begin{proof}
  We have $\tmop{Id}_{\mathbb{V}} + \tmop{Lin}_{\prec}^+ (\mathbb{V}) = \exp
  (\tmop{Lin}_{\prec}^+ (\mathbb{V}))$ by (\ref{eq-iso}), so the map
  $\Psi$ si uniquely determined by
  \[ \Psi (\sigma) \assign \exp (\Phi (\log (\sigma))) \]
  for all $\sigma \in \tmop{Id}_{\mathbb{V}} + \tmop{Lin}_{\prec}^+
  (\mathbb{V})$. Let $\sigma_1, \sigma_2 \in \tmop{Id}_{\mathbb{V}} +
  \tmop{Lin}_{\prec}^+ (\mathbb{V})$ and write $(\phi_1, \phi_2) \assign (\log
  (\sigma_1), \log (\sigma_2)) \in \tmop{Der}_{\prec}^+ (\mathbb{A})$. Since
  $\Phi$ is a morphism of Lie algebras, we have $\Phi (\tmop{ev}_{(\sigma_1,
  \sigma_2)} (Z_n)) = \tmop{ev}_{(\Phi (\sigma_1), \Phi (\sigma_2))} (K_n)$
  for all $n \in \mathbb{N}$, where $K_n \in k \langle \! \langle 2\rangle \! \rangle$ is as described in
  (\ref{eq-BCH-term}). Since $\Phi$ is strongly linear, we deduce that $\Phi
  (\tmop{ev}_{(\phi_1, \phi_2)} (X_0 \ast X_1)) = \tmop{ev}_{(\Phi (\phi_1),
  \Phi (\phi_2))} (X_0 \ast X_1)$, i.e.
  \begin{equation}
    \Phi (\phi_1 \ast \phi_2) = \Phi (\phi_1) \ast \Phi (\phi_2) .
    \label{eq-aux-BCH}
  \end{equation}
  It follows that
  \begin{align*}
    \Psi (\sigma_1 \circ \sigma_2) & = \Psi (\exp (\phi_1 \ast \phi_2)) &&
    \text{(\Cref{cor-BCH-operation})}\\
    & = \exp (\Phi (\phi_1 \ast \phi_2)) && \\
    & = \exp (\Phi (\phi_1) \ast \Phi (\phi_2)) && \text{(\ref{eq-aux-BCH})}\\
    & = \exp (\Phi (\phi_1)) \circ \exp (\Phi (\phi_2)) &&
    \text{(\Cref{cor-BCH-operation})}\\
    & = \Psi (\sigma_1) \circ \Psi (\sigma_2). &&
  \end{align*}
  Therefore $\Psi$ is a morphism.
\end{proof}

The same arguments using the identity $1 \text{-$\tmop{Aut}^+_k (\mathbb{A})$} =
\exp (\tmop{Der}_{\prec}^+ (\mathbb{A}))$ give:

\begin{theorem}
\label{th-Lie}
  Let $\mathbb{B}$ be an angelra of Noetherian series. Let $\Phi :
  \tmop{Der}^+_{\prec} (\mathbb{A}) \longrightarrow \tmop{Der}^+_{\prec}
  (\mathbb{B})$ be a strongly linear morphism of Lie algebras. Then there
  exists a unique group morphism $\Psi \colon 1 \text{-$\tmop{Aut}^+_k (\mathbb{A})$}
  \longrightarrow 1 \text{-$\tmop{Aut}^+_k (\mathbb{B})$}$ with
  \[ \Psi (\exp (\partial)) = \exp (\Phi (\partial)) \]
  for all $\partial \in \tmop{Der}^+_{\prec} (\mathbb{A})$.
\end{theorem}

As a consequence of \Cref{cor-divisible-torsion-free}, we have:

\begin{corollary}
\label{cor-divisible}
  The group $1
  \text{-$\tmop{Aut}^+_k (\mathbb{A})$}$ is divisible and torsion-free.
\end{corollary}

\subsubsection*{Acknowledgements}
    This project was initiated at the Fields Institute in June 2022 and resumed at the Banff International Research Station in February 2024 as well as \textsl{ddg40 : Structures algébriques et ordonnées} at Observatoire Océanologique de Banyuls-sur-Mer in August 2025.
    We wish to thank all institutions for their hospitality.
    Lothar Sebastian Krapp was partially supported by the Messmer Stiftung, the Vector Stiftung and the Network Platform \textit{Connecting Statistical Logic, Dynamical Systems and Optimization},
    Salma Kuhlmann and Michele Serra by the AFF of the University of Konstanz, and
    Daniel Panazzolo by the Project NonSper ANR-23-CE40-0028.

	\Addresses

\newpage 
\section*{Appendix: Toward a full correspondence}\label{Appendix}

\subsection*{Decomposing valuation preserving automorphisms}

Let $k$ be a (linearly) ordered field and let $G$ be a linearly ordered abelian group. There is a natural ordering on the field $\mathbb{K}\assign k(\!(G)\!)$ for which it is an ordered field \cite{Kuhl:Buch}. The field $\mathbb{K}$ is a Hahn field as per {\cite{Serra:phd,KuhlSer:vaut}}.

Let $v$-$\tmop{Aut}^+_k(\mathbb{K})$ denote the group of strongly linear automorphisms of the algebra $\mathbb{K}$ with $a\asymp b \Rightarrow \sigma(a) \asymp \sigma(b)$ for all $a,b \in U(\mathbb{K})$ where $a \asymp b$ if and only if  $a \not\prec b$ and $b \not\prec a$. By {\cite[Theorem~3.7.1]{KuhlSer:vaut}}, the group $v$-$\tmop{Aut}^+_k(\mathbb{K})$ is the semi-direct product of the following three subgroups:

\begin{enumeratealpha}
  \item The group $1$-$\tmop{Aut}_k^+ (\mathbb{K})$.
  
  \item The group $G$-$\operatorname{Exp}(\mathbb{K})$ of maps \[\Psi_x \colon a \mapsto \sum_{g \in G} x (g)
  t^g\] where $x \in \tmop{Hom} ((G, +), (k^{\times}, \cdot))$.
  
  \item The group $o$-$\tmop{Aut}(G)$ of maps \[a \mapsto \sum_{g \in G} a (g) t^{\mu (g)}\]
  where $\mu \in \tmop{Aut} (G, +, <)$.
\end{enumeratealpha}

\subsection*{Prelogarithms}

Write $\mathbb{K}_{\succ}\assign\{a \in \mathbb{K} \ \colon \ \tmop{supp} a <0\}$. The linearly ordered group $(\mathbb{K},+,0)$ has {\cite[Theorems~1.4 and~1.8]{Kuhl:Buch}} an additive lexicographic decomposition \[\mathbb{K} = \mathbb{K}_{\succ} + k + \mathbb{K}^{\prec},\] whereas $(\mathbb{K}^{>0},\cdot,1,<)$ has a multipicative decomposition \[\mathbb{K}^{>0} = t^G \cdot k^{>0} \cdot (1+\mathbb{K}^{\prec}).\]

A prelogarithm is an embedding $\log \colon(\mathbb{K}^{>0},\cdot,1,<) \longrightarrow (\mathbb{K},+,0,<)$. It is compatible with the valuation $v$ on $\mathbb{K}$ if for all $a >0$, we have $v (\log(a)) \geq 0 \Longleftrightarrow v(a) \geq 0$ and $v(\log(a))>0 \Longleftrightarrow v(a-1) >0$. A logarithm is a surjective prelogarithm. The existence of a compatible logarithm is equivalent {\cite[Lemma~1.21]{Kuhl:Buch}} to the existence of the three following isomorphisms of ordered groups.

\begin{enumeratealpha}
  \item A right logarithm, i.e.\ an isomorphism $(1 +\mathbb{K}^{\prec}, \cdot, 1, <) \longrightarrow (\mathbb{K}^{\prec}, +, 0,
  <)$.
  
  \item A middle logarithm, i.e.\ an isomorphism $(k^{> 0}, \cdot, 1, <) \longrightarrow (k, +, 0, <)$.
  
  \item A left logarithm, i.e.\ an isomorphism $(t^G, \cdot, 1, <) \longrightarrow (\mathbb{K}_{\succ}, +, 0, <)$.
\end{enumeratealpha}
This is illustrated in the following picture
{\center{\raisebox{-0.5\height}{\includegraphics[width=12.9077298963663cm,height=4.95805457169094cm]{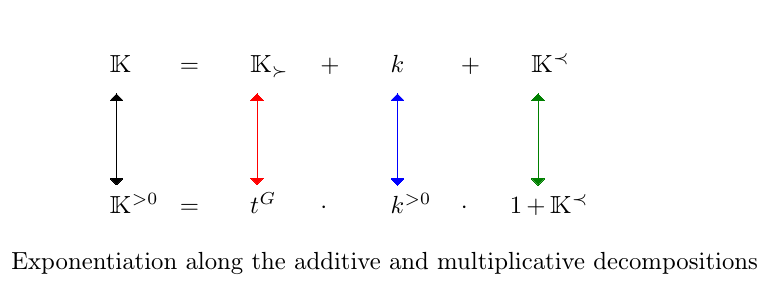}}}}

Although prelogarithms always exist, no right logarithm exsits if $G$ is non-trivial {\cite[Theorem~1]{KuhlShelah}}. This obstruction can be circumvented by considering directed unions of fields of Hahn series, such as $\log$-$\exp$ series, or fields of EL-series {\cite[Chap~5, Section~2]{Kuhl:Buch}}.

One can interpret the formal $\exp$-$\log$ correspondence between each $\tmop{Der}_{\prec}^+(\mathbb{K})$ and $1$-$\tmop{Aut}_k^+(\mathbb{K})$ as a non-commutative generalisation of the natural right logarithm on $\mathbb{K}$, which is the isomorphism \[(1+\mathbb{K}^{\prec},\cdot,<) \longrightarrow (\mathbb{K}^{\prec} ,+,<)  \ ; \ 1+\varepsilon \mapsto \sum \limits_{m \in \mathbb{N}} \frac{(-1)^m\varepsilon^{m+1}}{m}.\]

We want to investigate how this correspondence extends to the group $v$-$\tmop{Aut}^+_k(\mathbb{K})$, mapping it to an appropriate Lie subalgebra of $\tmop{Der}^+(\mathbb{K})$. Let us first extend the correspondence to the subgroup $\tmop{IntAut}^+_k(\mathbb{K})\assign G$-$\tmop{Exp}(\mathbb{K}) \ltimes 1$-$\tmop{Aut}^+_k(\mathbb{K})$ on $\mathbb{K}$.

\subsection*{The middle correspondence}

Let $G$-$\tmop{Der} (\mathbb{K})$ denote the commutative group, under
pointwise sum, of (strongly linear) derivations $\mathd_{\alpha}$ of the form
\[ \forall a \in \mathbb{K}, \mathd_{\alpha} (\sum_{g \in G} a (g) t^g) =
   \sum_{g \in G} \alpha (g) a (g) t^g \]
where $\alpha \in \tmop{Hom} ((G, +), (k, +))$. 

Note that $\tmop{Hom} ((G, +), (k, +))$ and
$G$-$\tmop{Der} (\mathbb{K})$ are vector spaces and $\mathd_{\cdot}$ is a $k$-linear isomorphism. Let
$\mathe \colon (k, +, <) \longrightarrow (k^{> 0}, \cdot, <)$ be an isomorphism of ordered groups, i.e.\ the inverse of a middle logarithm on $\mathbb{K}$.
Writing $k^{\times}$ as the direct product $k^{\times}\simeq \{1, - 1\}
\times k^{> 0}$, we obtain
\begin{eqnarray*}
  \tmop{Hom} ((G, +), (k^{\times}, \cdot)) & \simeq & \tmop{Hom} ((G, +),
  (\{1, - 1\}, \cdot)) \times \tmop{Hom} ((G, +), (k^{> 0}, \cdot))\\
  & \simeq & \tmop{Hom} ((G, +), (\mathbb{Z}_2, +)) \times \tmop{Hom} ((G,
  +), (k, +))
\end{eqnarray*}
We have $\tmop{Hom} ((G, +), (\mathbb{Z}_2, +)) \simeq \tmop{Hom} ((G / 2 G,
+), (\mathbb{Z}_2, +))$ which is trivial if and only if $G$ is $2$-divisile.
In that case, we have an isomorphism
\[ \tmop{Hom} ((G, +), (k, +)) \longrightarrow \tmop{Hom} ((G, +),
   (k^{\times}, \cdot)) \: ; \: \alpha \mapsto \mathe \circ \alpha . \]
which yields a Der-Aut correspondence $G$-$\tmop{Der}(\mathbb{T}) \longrightarrow G$-$\tmop{Exp}(\mathbb{T}) \ ;\ \mathd_{\alpha} \mapsto \psi_{\mathe \circ \alpha}$.

\subsection*{Toward a left correspondence}

Consider the Lie algebra $\tmop{Der}^+_{\preceq}(\mathbb{K})$ of strongly linear derivations $\partial \colon \mathbb{K} \longrightarrow \mathbb{K}$ such that $v(\partial(a))\geq v(a)$ for all $a \in \mathbb{K}$. This is the semi-direct product of the Lie algebra $G$-$\tmop{Der}(\mathbb{K})$ with the Lie ideal $\tmop{Der}^+_{\prec}(\mathbb{K})$. Combining the right and middle $\tmop{Der}$-$\tmop{Aut}$ correspondences, we obtain an isomorphism between $\tmop{Der}^+_{\preceq}(\mathbb{K})$ and the group $\tmop{IntAut}_k^+(\mathbb{K})$ of internal strongly linear $v$-automorphisms of $\mathbb{K}$. This is summed up in the following
picture, where the upper (respectively lower) decomposition is given by semi-direct
internal products of groups (resp Lie algebras).

{\center{\raisebox{-0.5\height}{\includegraphics[width=12.9077298963663cm,height=4.95805457169094cm]{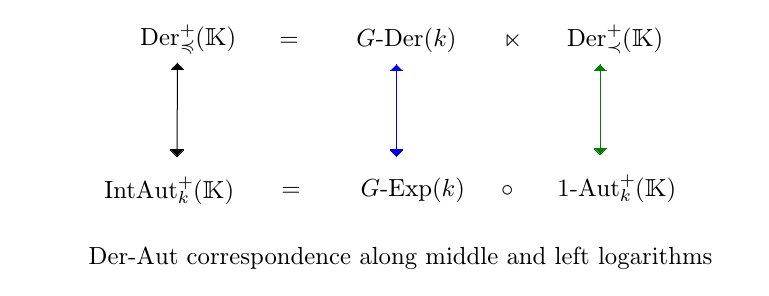}}}}

What of the left part of the decomposition of $v$-$\tmop{Aut}(\mathbb{K})$? Under what conditions on $\mathbb{K}$ is there a corresponding Lie algebra $\mathbf{D}$ of strongly linear derivations such that $\tmop{Der}^+_{\preceq}(\mathbb{K})$ is a Lie ideal of $\mathbf{D}+\tmop{Der}^+_{\preceq}(\mathbb{K})$ and that there exists a bijective correspondence between $\mathbf{D}$ and $o$-$\tmop{Aut}(G)$?

 In some cases, automorphisms of $(G,+,<)$ have been shown {\cite[Proposition~4.9]{KuhlMat:1}} to induce strongly linear derivations on $\mathbb{K}$. Furthermore, the derivations can be chosen compatible with a specific prelogarithm on $\mathbb{K}$ {\cite[Section~3]{KuhlMat:2}}.

\medskip

\noindent\textbf{Future work.} We plan to investigate this construction method in order to obtain further derivations coming from automorphisms on $G$.

\end{document}